\documentclass[a4j,11pt]{amsart}
\usepackage{amsmath,amssymb,amscd,color,mathrsfs}
\usepackage[all,arc,curve,color,frame,poly]{xy}
\usepackage{enumerate}
\usepackage{graphicx}

\theoremstyle{definition}
\numberwithin{equation}{section}
\newtheorem{thm}{Theorem}[section]
\newtheorem{df}[thm]{Definition}
\newtheorem{ex}[thm]{Example}
\newtheorem{prop}[thm]{Proposition}
\newtheorem{cor}[thm]{Corollary}
\newtheorem{lem}[thm]{Lemma}

\newtheorem{rmk}[thm]{Remark}

\newtheorem*{ack}{Acknowledgements}
\newcommand{\udim}{\underline{\dim}}
\newcommand{\IM}{\mathrm{Im}}
\newcommand{\Ker}{\mathrm{Ker}}
\newcommand{\Coker}{\mathrm{Coker}}
\newcommand{\Hom}{\mathrm{Hom}}
\newcommand{\End}{\mathrm{End}}
\newcommand{\Ext}{\mathrm{Ext}}
\newcommand{\Tor}{\mathrm{Tor}}
\newcommand{\Spec}{\mathrm{Spec}}
\newcommand{\Mod}{\mathrm{Mod}}
\newcommand{\RHom}{\mathbb R\mathrm{Hom}}
\newcommand{\Soc}{\mathrm{Soc}}
\newcommand{\add}{\mathrm{add}}
\newcommand{\nilp}{\mathrm{nilp}}
\newcommand{\diag}{\mathrm{diag}}
\newcommand{\op}{\mathrm{op}}
\newcommand{\fd}{\mathrm{fd}}

\newcommand{\e}{{\mathbf{e}}}
\newcommand{\s}[3]{\mathcal S_{#1 #3}(#2)}
\newcommand{\Ltensor}{\stackrel{\mathbb L}{\otimes}}

\newcommand{\dcat}[2]{\mathcal{D}_{#1}(#2)}
\newcommand{\uM}[3]{h_{\mathcal M_{#1,#3}(#2)}}
\newcommand{\M}[3]{\mathcal M_{#1,#3}(#2)}
\newcommand{\Ph}[3]{\Phi_{#1,#3,#2}}
\newcommand{\F}[3]{\mathcal F_{#1,#3,#2}}
\newcommand{\Max}{\mathrm{Max}}
\newcommand{\m}{{\mathfrak{m}}}
\newcommand{\R}{{\mathfrak{R}}}
\newcommand{\Set}{{\mathfrak{Set}}}
\newcommand{\Var}{{\mathfrak{Var}}}
\newcommand{\bs}{\mathbf s}
\newcommand{\bw}{\mathbf w}
\newcommand{\ove}{\overline{\mathbf{e}}}

\title[Tilting theoretical approach to moduli spaces over preprojective algebras]{Tilting theoretical approach to moduli spaces over preprojective algebras}
\author{Yuhi Sekiya}
\address{Graduate School of Mathematics\\ 
 Nagoya University, Chikusa-ku Nagoya 464-8602 Japan}
\email{yuhi-sekiya@math.nagoya-u.ac.jp}
\author{Kota Yamaura}
\address{Graduate School of Mathematics\\ 
 Nagoya University, Chikusa-ku Nagoya 464-8602 Japan}
\email{m07052d@math.nagoya-u.ac.jp}
\thanks{The first author is supported by JSPS Fellowships for Young Scientists No.21-6922.
}
\thanks{The second author is supported by JSPS Fellowships for Young Scientists No.22-5801.
}

\begin{document}

\maketitle

\begin{abstract}
We apply tilting theory over preprojective algebras $\Lambda$ to a study of moduli space of $\Lambda$-modules.
We define the categories of semistable modules and give an equivalence, so-called reflection functors, between them by using tilting modules over $\Lambda$. Moreover we prove that the equivalence induces an isomorphism of algebraic varieties between moduli spaces.
In particular, we study in the case when the moduli spaces related to the Kleinian singularity. We generalize a result of Crawley-Boevey which is known another proof of the McKay correspondence of Ito-Nakamura type.
\end{abstract}

\tableofcontents

\section{Introduction}

In this paper we attempt to apply tilting theory of algebras to a study of moduli spaces of modules over algebras.
Throughput this paper $K$ denotes an arbitrary algebraically closed field. For any $K$-algebra $\Lambda$, $\Mod\Lambda$ denotes the category of right $\Lambda$-modules.

The motivation is a study of moduli spaces of modules over algebras related to singularities on algebraic varieties.
We especially focus on the quotient singularity. When we discuss about the quotient singularity, we always assume that the characteristic of $K$ is zero.
In this paper, we consider the 2-dimensional case.
Let $G$ be a finite subgroup of $SL(2,K)$.
Then $G$ acts naturally on an affine space $\mathbb A^2 = \mathbb A^2_K$ and its coordinate ring $S=K[\mathbb A^2]$. 
The quotient singularity $\mathbb A^2/G$ has a lot of names like Kleinian singularity, Gorenstein quotient singularity, rational double point and so on. Since the observation of McKay \cite{M}, a field of singularities related to the McKay correspondence has been developed. 
In particular, after Ito-Nakamura \cite{IN} introduced the $G$-Hilbert scheme, moduli spaces of $G$-clusters, $G$-sheaves, $G$-constellations or McKay quiver representations have been investigated in connection with resolutions of singularities by many researchers (e.g.~\cite{BKR,CI} etc.). 
On the other hand, in representation theory of the invariant ring $S^G$, the skew group ring $S*G$ has been studied (cf.~\cite{A,AR}). 
The non-commutative ring $S*G$ is also called a non-commutative crepant resolution in \cite{VdB,IW} because it is considered as a non-commutative analogue of the crepant resolution. 
In fact, all $G$-clusters, $G$-sheaves etc. are regarded as modules over the skew group ring $S*G$, so the resolution of the singularity is recovered as a moduli space of $S*G$-modules.

By the way, the skew group ring $S*G$ is Morita equivalent to the preprojective algebra $\Lambda = K\overline{Q}/\langle R \rangle$ of an extended Dynkin quiver $Q$ (cf.~\cite{R-VdB,BSW}). Here $\overline{Q}$ is the double of $Q$ and $R$ is the preprojective relation.
In this paper, therefore, we deal with modules over any preprojective algebra $\Lambda$ which is not necessarily Morita equivalent to $S*G$.
Since the category of $\Lambda$-modules is equivalent to the category of representations of $(\overline{Q},R)$ (cf.~\cite[Chapter 3]{ASS}), we often confuse $\Lambda$-modules and representations of $(Q,R)$.
By virtue of King \cite{Ki}, moduli spaces of $\Lambda$-modules are constructed by using geometric invariant theory. 
We identify dimension vectors with elements in $\mathbb Z^{Q_0}$ and denote by $(\mathbb Z^{Q_0})^*$ the dual lattice. We call $\Theta = (\mathbb Z^{Q_0})^{*}\otimes_{\mathbb Z}\mathbb Q$ the parameter space.  
For any $\theta \in \Theta$, we denote by $\M{\theta}{\Lambda}{\alpha}$ the moduli space of $\theta$-semistable $\Lambda$-modules of dimension vector $\alpha$, which is actually a coarse moduli space parametrizing S-equivalence classes of $\theta$-semistable modules of dimension vector $\alpha$.

The purpose of this paper is to study a relation between variation of parameter $\theta$ and variation of moduli space $\mathcal M_{\theta}(\Lambda,\alpha)$ by using tilting theory over the preprojective algebra $\Lambda$. 
We want to deal with $\theta$-semistable $\Lambda$-modules categorically, hence we give the following definition. 

\begin{df}
For any parameter $\theta \in \Theta$, we define the full subcategory $\s{\theta}{\Lambda}{}$ of $\Mod \Lambda$ consisting of $\theta$-semistable $\Lambda$-modules. Moreover we denote by $\s{\theta}{\Lambda}{,\alpha}$ the full subcategory of $\s{\theta}{\Lambda}{}$ consisting of $\theta$-semistable $\Lambda$-modules of dimension vector $\alpha$ if $\s{\theta}{\Lambda}{,\alpha}$ is not empty. 
\end{df}

The tilting theory is a theory to deal with equivalences of derived categories of modules over algebras, which is applied to fields around the representation theory of algebras (cf.~\cite{AHK,Hap}).
Tilting modules, or more generally tilting complexes, play an important roll in the tilting theory.
If a tilting module $T$ is given, then the derived category $\mathcal D(\Mod \Lambda)$ becomes equivalent to the derived category $\mathcal D(\Mod \End_{\Lambda}(T))$ (cf.~\cite{Ri}).
On the other hand, preprojective algebras were introduced by Gelfand and Ponomarev \cite{GP} and have been studied by many researchers (for example \cite{Boc,BGL,BIRS}).
Buan, Iyama, Reiten and Scott \cite{BIRS} studied tilting theory on preprojective algebras.
They constructed a set of tilting modules over preprojective algebras  of projective dimension at most one as follows. 
For each vertex $i \in Q_0$ with no loops, they defined the two sided ideal $I_i$ and proved that any products of these ideals are tilting modules. We denote by $\mathcal I(\Lambda)$ the set of such tilting modules.
Moreover they showed that, for any $T \in \mathcal I(\Lambda)$ the endomorphism ring $\End_{\Lambda}(T)$ is isomorphic to the original ring $\Lambda$, and $\mathcal{I}(\Lambda)$ is related to the Coxeter group $W_Q$ associated to $Q$, that is, there exists a bijection $W_Q \longrightarrow \mathcal{I}(\Lambda) \ ;\ w \longmapsto I_{w}$ where $I_w$ is well-defined as the product $I_{i_1} \cdots I_{i_{\ell}}$ for any reduced expression $w=s_{i_1}\cdots s_{i_{\ell}}$.

For a tilting module $I_w$, we have the derived auto-equivalence 
\[\xymatrix{
\mathcal D(\Mod \Lambda) \ar@<0.5ex>[rrr]^{\RHom_{\Lambda}(I_w,-)} 
\ar@<-0.5ex>@{<-}[rrr]_{-\Ltensor_{\Lambda}I_w} &&& \mathcal D(\Mod \Lambda)
}\]
which induce the main result in this paper. The coxeter group $W_Q$ acts on both of the set of dimension vectors and the parameter space $\Theta$.

\begin{thm}[Theorem \ref{simple-reflection} and \ref{reflection-functor}] \label{1}
For any preprojective algebra $\Lambda$ and any $\theta \in \Theta$ with $\theta_i>0$, there is a categorical equivalence
\[
\xymatrix{
\s{\theta}{\Lambda}{} \ar@<0.5ex>[rr]^{\Hom_{\Lambda}(I_i,-)} & & \s{s_i \theta}{\Lambda}{} \ar@<0.5ex>[ll]^{- \otimes_{\Lambda} I_i}
}.
\]
For any element $w \in W_Q$ and any sufficiently general parameter $\theta \in \Theta$, an equivalence between $\s{\theta}{\Lambda}{}$ and $\s{w\theta}{\Lambda}{}$ is given by a composition of them.
Moreover it preserves stable objects and $S$-equivalence classes. Moreover it induces the equivalence between $\s{\theta}{\Lambda}{,\alpha}$ and $\s{w\theta}{\Lambda}{,w\alpha}$.
\end{thm}

We call these functors the reflection functors.
The above equivalence induces a bijection on sets of closed points of $\M{\theta}{\Lambda}{\alpha}$ and $\M{w\theta}{\Lambda}{w\alpha}$.
It is natural to hope that this bijection is extended to an isomorphism of algebraic varieties.
In this paper, we prove it in a more general setting in view of applications: for example, $\Lambda$ is a $d$-Calabi-Yau algebra and $L$ is a partial tilting $\Lambda$-module and so on.

\begin{thm}[Theorem \ref{equiv-iso}]\label{i}
For a finite quiver $(Q,R)$ with a relation, let $\Lambda=KQ/\langle R \rangle$. 
For a $\Lambda$-module $L$ which satisfies an appropriate condition (see Proposition \ref{hom}), we assume that $\Gamma = \End_{\Lambda}(L)$ is of the form $\Gamma=KQ'/\langle R' \rangle$ for a finite quiver $(Q',R')$ with a relation.
If the functors $\Hom_{\Lambda}(L,-)$ and $-\otimes_{\Gamma}L$ give a categorical equivalence 
\[\xymatrix{
\s{\theta}{\Lambda}{,\alpha} \ar@<0.5ex>[rr]^{\Hom_{\Lambda}(L,-)} && \ar@<0.5ex>[ll]^{-\otimes_{\Gamma}L} \s{\eta}{\Gamma}{,\beta}
}\]
which preserve $S$-equivalence classes, then they are extended to morphisms of algebraic varieties and give an isomorphism of algebraic varieties:
\[
\M{\theta}{\Lambda}{\alpha} \stackrel{\sim}{\longrightarrow} \M{\eta}{\Gamma}{\beta}.
\]
\end{thm}

We remark that Nakajima \cite{N} has already considered about similar isomorphisms between quiver varieties, which is a space of representations over a deformed preprojective algebra. But it is defined as a hyper-K\"ahler quotient and his method is differential geometric.
Maffei \cite{Maf} also study them by using technics in geometric invariant theory. 
So our result is more or less known. However we define the reflection functor by using unconventional tactics and the fact that it is induced by the equivalence between derived categories has never known. Furthermore since our reflection functor is realized as a equivalence of categories, it gives the correspondence between not only the objects but also morphisms between them.
\\

Now we return to the Kleinian singularity case, that is, the case when $\Lambda$ is Morita equivalent to the skew group ring $S*G$.
Let $Q$ be the extended Dynkin quiver whose type is the same of that of $G$. Then the double $\overline{Q}$ of $Q$ coincides with the Mckay quiver of $G$. 
Denote the vertexes of $Q$ by $0,1,\ldots,n$ where $0$ corresponds to the trivial representation of $G$. 
We fix the dimension vector $\mathbf d$ whose entries are the dimension of irreducible representations of $G$.
Let $\Theta_{\mathbf d}$ be a subset of $\Theta$ consisting of parameters with $\theta(\mathbf d)=0$.
Then it is known by  \cite{Kr,CS,BKR} that, if $\theta \in \Theta_{\mathbf d}$ is generic, then $\M{\theta}{\Lambda}{\mathbf d}$ gives a minimal resolution of $\mathbb A^2/G$. 
Since we especially interested in the minimal resolution, we only consider generic parameters $\theta$ in $\Theta_{\mathbf d}$.  
Then categories $\s{\theta}{\Lambda}{}$ are classified by the chamber structure of $\Theta_{\mathbf d}$.
If we denote the generic locus of $\Theta_{\mathbf d}$ by $\Theta_{\mathbf d}^{\mathrm{gen}}$, then it decomposes into connected components by using the elements of the Weyl group $W$ whose type is the same as that of $G$: $\Theta_{\mathbf d}^{\mathrm{gen}} = \prod_{w \in W} C(w)$.  
Thus for a $\theta \in C(w)$, we just denote $\s{\theta}{\Lambda}{,\mathbf d}$ and $\M{\theta}{\Lambda}{\mathbf d}$ by $\mathcal S_w$ and $\mathcal M_w$ respectively.
For the identity element $1 \in W$, Crawley-Boevey \cite{CB} observed that the $G$-Hilbert scheme is naturally identified with $\mathcal M_1$ via the Morita equivalence between $S*G$ and $\Lambda$.

The purpose in this case is to study properties of stable modules.
As an application of Theorem \ref{1} and \ref{i}, we give a module theoretical description of stability condition and exceptional curves.
The complexes $S_i^w = \RHom_{\Lambda}(I_w,S_i)$ defined in Section \ref{S_i^w} play an important role in the arguments. It is actually a complexes concentrated in degree 0 or 1 and $[S_1^w],\ldots,[S_n^w]$ form the simple root system $w\Delta$ for any $w \in W$, where $\Delta := \{\ove_1,\ldots, \ove_n \}$ is a simple root system and $\e_i$ denotes the dimension vector of $S_i$ for any $i \in Q_0$ (see Section \ref{Klein} for details). We denote by $\Pi$ the positive root system associated to $\Delta$. Then we have the following result.

\begin{thm}[Proposition \ref{E_i^w} and Theorem \ref{thm:ch-s}]\label{2}
Let $M$ be a $\Lambda$-module of dimension vector $\mathbf d$. Then $M \in \mathcal S_w$ if and only if, $M$ satisfies
\[
\begin{cases}
\Hom_{\Lambda}(M,S_i^w)=0 & \text{ if $w\e_i \in \Pi$},\\
\Hom_{\Lambda}(S_i^w[1],M)=0 & \text{ if $w\e_i \in -\Pi$}
\end{cases}
\]
for all $i=1,\ldots,n$.
\end{thm}

Next we give a characterization of exceptional curves on $\mathcal M_w$. This is a generalization of a result of Crawley-Boevey \cite{CB} known as another proof of the McKay correspondence of Ito-Nakamura type \cite{IN}.
For any $i \in \{ 1,\ldots,n \}$, let $E_i := \{ M \in \mathcal M_{1} \mid S_i \subset M \}$ be a closed set of $\mathcal M_{1}$ and $E_i^w$ is the image of $E_i$ under the isomorphism $\mathcal M_1 \to \mathcal M_w$ obtained by Theorem \ref{1} and \ref{i}.

\begin{thm}[Theorem \ref{hom-des-main}]\label{3}
If $i \in \{ 1,\ldots, n\}$, then the set $E_i^w$ is a closed subset of $\mathcal M_w$ isomorphic to $\mathbb P^1$. Moreover $E_i^w$ meets $E_j^w$ if and only if $i$ and $j$ are adjacent in $Q$, and in this case they meet at only one point. Moreover $M \in E_i^w$ if and only if 
\[
\begin{cases}
S_i^w \text{ is a submodule of } $M$ & \text{if $w\e_i \in \Pi$ \quad or} \\
S_i^w[1] \text{ is a factor module of }  M & \text{if $w\e_i \in -\Pi$}.
\end{cases}
\]
\end{thm}

We explain the contents of each sections.
In Section \ref{tilting}, we recall preprojective algebras and their tilting theory  shown in \cite{BIRS}. Since we are dealing with non-completed algebras, we give a proof for readers. Moreover we study the change of dimension vectors and the structure of finite dimensional modules over the preprojective algebra associated to an extended Dynkin quiver.
In Section \ref{moduli}, we prepare necessary notations about moduli spaces of modules.
In Section \ref{reflection}, we prove Theorem \ref{1}. The arguments are separated into two parts. First we consider the simple reflection case. Second we prove the simple reflection functor satisfies the Coxeter relation in the general case. 
In Section \ref{functor}, we prove Theorem \ref{i}.
In Section \ref{Klein}, we study the reflection functor in this case and prove Theorem \ref{2} and \ref{3}.
Finally in Section \ref{example}, we give an example in the Kleinian singularity case with $G$ is the abelian group of order three.

\begin{ack}
The authors would like to thank Professor O.~Iyama for a lot of kindly and useful advices and giving us lectures on preprojective algebras. The first author is grateful to Professor Y.~Ito for giving him a chance to study McKay correspondence and related topics and her warm encouragement.
They also thank K.~Nagao for useful advices about the proof of Theorem \ref{equiv-iso}. They also thank A.~Craw, Y.~Kimura and M.~Wemyss for useful comments. They thank T.~Hayashi for pointing out the lack of an assumption in Lemma \ref{fiber}. They also thank A.~Nolla de Celis for his kind advice in correcting our English.
\end{ack}

\section{Tilting theory on preprojective algebras}
\label{tilting}

In this section, we study representation theory of preprojective algebras which plays a key role in this paper. 
In Section \ref{pp} we recall the definition of preprojective algebras and its basic facts. 
In Section \ref{pptilting},  we construct a set of tilting modules over preprojective algebras whose endomorphism ring is isomorphic to the original algebra,
and in \ref{Coxeter} show that it is described by the Coxeter group.
Almost results shown in those sections were  proved in \cite{BIRS}. 
In Section \ref{dimension-tilting} we study the change of dimension vectors of finite dimensional modules under the derived equivalence which are induced by tilting modules constructed in Section \ref{pptilting}.
In Section \ref{torsion-tilting} we observe torsion pairs over finite dimensional modules.
In Section \ref{fdDynkin} we study the structure of the category of finite dimensional modules over preprojective algebras of extended Dykin quivers.

For an algebra $\Lambda$, we always deal with right $\Lambda$-modules.
We denote by $\mathrm{Mod}\Lambda$ the category of all right $\Lambda$-modules and
$\dcat{}{\mathrm{Mod}\Lambda}$ the derived category of $\mathrm{Mod}\Lambda$. 
For a $\Lambda$-module $M$, we denote by $\mathrm{add}M$ the full subcategory of $\mathrm{Mod}\Lambda$ whose objects consist of 
direct summands of direct sums of finite copies of $M$.
We put $D:=\Hom_K(-,K)$ the $K$-dual.

\subsection{Preprojective algebras}
\label{pp}

In this subsection, we recall the definition of preprojective algebras and properties of preprojective algebras of non-Dynkin quivers.

A \emph{quiver} is a quadruple $Q = (Q_0,Q_1,s,t)$ which consists of a vertex set $Q_0$ and an arrow set $Q_1$, and maps $s,t : Q_1 \to Q_0$ 
which associate to each arrow $a \in Q_1$ its source $sa := s(a) \in Q_0$ and its target $ta := t(a) \in Q_0$, respectively. 
We call $a \in Q_1$ a \emph{loop} if $sa=ta$.
A quiver is called non-Dynkin if its underlying graph is not a Dykin graph.

\begin{df}\label{def-pp}
Let $Q$ be a finite connected quiver. We define the double quiver $\overline{Q}$ of $Q$ by
\[
\overline{Q}_0:=Q_0
\]
and
\[
\overline{Q}_1:=Q_1 \bigsqcup \left\{  j \xrightarrow{\alpha^*} i \ | \ i \xrightarrow{\alpha} j \in Q_1 \right\}.
\]
Then we have a bijection $^*:  \overline{Q}_1\longrightarrow \overline{Q}_1$ which is defined by
\[
\alpha^* := \begin{cases}
\alpha^* & (\alpha \in Q_1), \\
\beta & (\alpha=\beta^*  \mbox{ for some } \beta \in Q_1).
\end{cases}
\]

We define a relation $\rho_i$ for any $i \in \overline{Q}_0 $ by
\[
\rho_i := \sum_{i \xrightarrow{\alpha} j \in \overline{Q}_1} \epsilon_{\alpha}  \alpha \alpha^*
\]
where 
\[
\epsilon_{\alpha} := \begin{cases}
1 & (\alpha \in Q_1), \\
-1 & (\alpha^* \in Q_1).
\end{cases}
\]
A relation $\sum_{i \in \overline{Q}_0} \rho_i$ is called a \emph{preprojective relation}.
We call an algebra
\[
K\overline{Q}/\langle \rho_i \ | \ i \in \overline{Q}_0 \rangle
\]
\emph{the preprojective algebra of $Q$}.
\end{df}

\begin{rmk} 
We give two remarks.
\begin{enumerate}
\def\labelenumi{(\theenumi)} 
\item Let $Q$ and $Q'$ be quivers which have the same underlying graph. 
Then the preprojective algebra of $Q$ and that of $Q'$ are isomorphic to each other as $K$-algebras.
\item The preprojective algebra of $Q$ is not finite dimensional if and only if $Q$ is a non-Dynkin quiver.
\end{enumerate}
\end{rmk}

Throughout this section, let $\Lambda$ be the preprojective algebra of a finite connected non-Dynkin quiver $Q$ which has no loops with the vertex set $Q_0 = \{ 0,1,\ldots,n \}$.
We denote by $I$ the two-sided ideal of $\Lambda$ which is generated by all arrows in $\overline{Q}$, 
$e_i$ the primitive idempotent of $\Lambda$ which corresponds to a vertex $i \in \overline{Q}_0$, 
and $S_i$ the simple $\Lambda$-module which corresponds to a vertex $i \in \overline{Q}_0$. 

In this setting, simple $\Lambda$-modules $S_0,S_1,\ldots, S_n$ have projective resolutions, which plays a crucial role in the representation theory of $\Lambda$.

\begin{prop}[{\cite[Section 4.1]{BBK}}]\label{pp-Koszul}
For any $i \in \overline{Q}_0$, the following hold.
\begin{enumerate}
\def\labelenumi{(\theenumi)} 
\item A complex
\begin{eqnarray}
0 \longrightarrow e_i\Lambda \xrightarrow{(\epsilon_{\alpha}\alpha^*)_{i \xrightarrow{\alpha} j \in \overline{Q}_1}} \bigoplus_{i \xrightarrow{\alpha} j \in \overline{Q}_1}  e_j\Lambda \xrightarrow{(\alpha)_{i \xrightarrow{\alpha} j \in \overline{Q}_1}} e_i\Lambda \longrightarrow S_i \longrightarrow 0 \label{exact-ap1}
\end{eqnarray}
is a projective resolution of the right $\Lambda$-module $S_i$.
\item A complex
\begin{eqnarray}
0 \longrightarrow \Lambda e_i \xrightarrow{(\epsilon_{\alpha^*}\alpha^*)_{j \xrightarrow{\alpha} i \in \overline{Q}_1}} \bigoplus_{j \xrightarrow{\alpha} i \in \overline{Q}_1}  \Lambda e_j  \xrightarrow{(\alpha)_{j \xrightarrow{\alpha} i \in \overline{Q}_1}} \Lambda e_i \longrightarrow S_i \longrightarrow 0 \label{exact-ap2}
\end{eqnarray}
is a projective resolution of the left $\Lambda$-module $S_i$.
\end{enumerate}
\end{prop}

The following property is called $2$-Calabi-Yau property. 

\begin{lem}\label{lem:1-1}
There exists a functorial isomorphism
\[
\Hom_{\dcat{}{\mathrm{Mod}\Lambda}}(M,N) \simeq D \Hom_{\dcat{}{\mathrm{Mod}\Lambda}}(N,M[2]).
\]
for any $M \in \dcat{}{\mathrm{Mod}\Lambda}$ whose total homology is finite dimensional and any $N \in \dcat{}{\mathrm{Mod}\Lambda}$.
\end{lem}
\begin{proof}
See \cite[Theorem 9.2]{BK} and \cite[Lemma 4.1]{Ke}.
\end{proof}

Now we show a useful lemma.
The dimension vector $\udim M$ of a finite dimensional $\Lambda$-module $M$ is defined by
\[
\udim M := ^t(\dim (Me_0),\dim (Me_1),\ldots,\dim (Me_n)) \in \mathbb Z^{\overline{Q}_0}.
\]
Let $(-,-)$ be a symmetric bilinear form on $\mathbb Z^{\overline{Q}_0}$ defined by
\[
(\alpha,\beta) = \sum_{i \in \overline{Q}_0}2\alpha_i \beta_i - \sum_{a\in \overline{Q}_1} \alpha_{sa} \beta_{ta}.
\]
We define $(M,N):=(\udim M,\udim N)$ for any finite dimensional $\Lambda$-modules $M,N$.

\begin{lem}[{\cite[Lemma 1]{CB}}]\label{CB-form}
Let $M$ and $N$ be finite dimensional $\Lambda$-modules. Then the following holds.
\begin{eqnarray*}
(M,N) &=& \dim \Hom_{\Lambda}(M,N) - \dim \Ext^1_{\Lambda}(M,N) + \dim \Ext^2_{\Lambda}(M,N) \\
&=& \dim \Hom_{\Lambda}(M,N) - \dim \Ext^1_{\Lambda}(M,N) + \dim \Hom_{\Lambda}(N,M).
\end{eqnarray*}
\end{lem}

The following holds for general $K$-algebra $\Lambda$.

\begin{lem}[{\cite[Chapter VI Proposition 5.1]{CE}}]\label{isom1}
We have a functorial isomorphism
\[
\Ext^i_{\Lambda}(M,DN) \simeq D \Tor^{\Lambda}_i(M,N)
\]
for any $M \in \mathrm{Mod}\Lambda$, $N \in \mathrm{Mod}\Lambda^{\mathrm{op}}$ and $i \in \mathbb{N} \cup \{0\}$.
\end{lem}

\subsection{Tilting theory on preprojective algebras}
\label{pptilting}

In this subsection, we recall the construction of tilting modules over preprojective algebras of non-Dynkin quivers which was shown in \cite{IR,BIRS}, and properties of those tilting modules. 
We keep the notations in the previous subsection.

We start with recalling the definition of tilting modules.

\begin{df}\label{df_tilt}
A $\Lambda$-module $T$ is called a \emph{tilting module} if it satisfies the following conditions.
\begin{enumerate}
\def\labelenumi{(\theenumi)}
\item There exists an exact sequence
\begin{eqnarray}
0 \longrightarrow P_1 \longrightarrow P_0 \longrightarrow T \longrightarrow 0 \label{exact-fgp}
\end{eqnarray}
where $P_0,P_1 \in \add \Lambda$.
\item $\Ext^{1}_{\Lambda}(T,T)=0$.
\item There exists an exact sequence
\begin{eqnarray}
0 \longrightarrow \Lambda_{\Lambda} \longrightarrow T_0 \longrightarrow T_1 \longrightarrow 0 \label{exact-tilt}
\end{eqnarray}
where $T_0,T_1 \in \add T$.
\end{enumerate}
\end{df}

Let $T$ be a tilting $\Lambda$-module. We put $\Gamma:=\End_{\Lambda}(T)$. 
\cite{Ri} showed that $T$ induces a triangle equivalence
\[
\xymatrix{
\dcat{}{\mathrm{Mod}\Lambda} \ar@<0.5ex>[rr]^{\RHom_{\Lambda}(T,-)} & & \dcat{}{\mathrm{Mod} \Gamma} \ar@<0.5ex>[ll]^{-\Ltensor_{\Gamma} T}
}.
\]

Thus it is important for representation theory which has been developed by using derived categories to study constructions or classifications of tilting modules. 
The construction of tilting modules over preprojective algebras of non-Dynkin quivers was given by \cite{BIRS} as follows.

We define a two-sided ideal $I_i$ of $\Lambda$ by 
\[
I_i := \Lambda (1-e_i) \Lambda
\]
for any $i \in \overline{Q}_0$.
Then we have an exact sequence 
\begin{eqnarray}
0 \longrightarrow I_i \longrightarrow \Lambda \longrightarrow S_i \longrightarrow 0 \label{exact1}
\end{eqnarray}
of $(\Lambda,\Lambda)$-bimodules for any $i \in \overline{Q}_0$ since $\overline{Q}$ has no loops.
We consider a set
\[
\mathcal{I}(\Lambda):= \{ I_{i_1}I_{i_2} \cdots I_{i_{\ell}} \ | \ l \in \mathbb{N}\cup\{0\}, \ i_1,i_2,\ldots,i_{\ell} \in Q_0 \}
\]
where $I_{i_1}I_{i_2} \cdots I_{i_{\ell}}$ is an ideal which is obtained by product of ideals $I_{i_1},I_{i_2}, \ldots ,I_{i_{\ell}}$.
Then the following result holds.

\begin{thm}[{\cite[Proposition III.1.4. Theorem III.1.6.]{BIRS}}]\label{classi-tilt}
Any $T \in \mathcal{I}(\Lambda)$ is a tilting $\Lambda$-module with $\End_{\Lambda}(T) \simeq \Lambda$.
\end{thm}

In the following we give a proof of the above result since the setting in this paper is different from that of \cite{BIRS}.
Namely they dealt with completed preprojective algebras of non-Dynkin quivers, but we deal with non-completed one.

To prove Theorem \ref{classi-tilt}, we need the following lemma which was pointed out to us by Osamu Iyama and Idun Reiten.

\begin{lem}\label{IRprivate}
Let $T$ be a tilting $\Lambda$-module, S a simple $\Lambda^{\op}$-module.
Then exactly one of the statements $T\otimes_{\Lambda}S=0$ and $\Tor^{\Lambda}_{1}(T,S)=0$ holds.
\end{lem}
\begin{proof}
First we assume that $T \otimes_{\Lambda}S=0=\Tor^{\Lambda}_1(T,S)$. 
Then we have $T \Ltensor S=0$. But this is a contradiction since 
\[
T \Ltensor - : \dcat{}{\mathrm{Mod}\Lambda^{\op}} \longrightarrow \dcat{}{\mathrm{Mod}\End_{\Lambda}(T)^{\op}}.
\]
is a triangle equivalence.

Since $T$ is a tilting $\Lambda$-module, there exists an exact sequence
\[
0 \longrightarrow P_1 \xrightarrow{\ a \ } P_0 \xrightarrow{\ b \ } T \longrightarrow 0
\]
where $P_0,P_1 \in \mathrm{add}\Lambda$.
Now we claim that $\Hom_{\Lambda}(P_1,P_0)=a \End_{\Lambda}(P_1)+\End_{\Lambda}(P_0)a$.
We take $f \in \Hom_{\Lambda}(P_1,P_0)$.
By applying $\Hom_{\Lambda}(-,T)$ to the above exact sequence, we have an exact sequence
\[
\Hom_{\Lambda}(P_0,T) \longrightarrow \Hom_{\Lambda}(P_1,T) \longrightarrow \Ext^1_{\Lambda}(T,T)=0.
\]
Therefore there exists  $c \in \Hom_{\Lambda}(P_0,T)$ such that $ca=bf$. 
Since $P_0$ is projective, there exists $d \in \Hom_{\Lambda}(P_0,P_0)$ such that $c=bd$.
\[
\xymatrix{
& 0 \ar[r] & P_1 \ar[r]^a \ar[d]_f & P_0 \ar[r]^b \ar[d]^c \ar[ld]^d & T \ar[r] & 0 \\
0 \ar[r] & P_1 \ar[r]_a & P_0 \ar[r]_b & T \ar[r] & 0 &
}
\]
Since $b(f-da)=bf-bda=bf-ca=0$, there exists $e \in \Hom_{\Lambda}(P_1,P_1)$ such that $f-da=ae$. 
Thus we have $f \in a \End_{\Lambda}(P_1)+\End_{\Lambda}(P_0)a$.

Next we put $\Gamma:= \End_{\Lambda}(S)^{\op}$.
We consider the following commutative diagram.
\[
\xymatrix{
a \End_{\Lambda}(P_1)+\End_{\Lambda}(P_0) a \ar@{=}[r] \ar[d]_{- \otimes_{\Lambda}S} & \Hom_{\Lambda}(P_1,P_0) \ar[d]^{- \otimes_{\Lambda}S}\\
(a \otimes_{\Lambda}S) \End_{\Gamma}(P_1 \otimes_{\Lambda}S)+\End_{\Gamma}(P_0 \otimes_{\Lambda}S) (a \otimes_{\Lambda}S) \ar@{^{(}->}[r] & 
\Hom_{\Gamma}(P_1 \otimes_{\Lambda}S,P_0 \otimes_{\Lambda}S)
}
\]
Since the right vertical map is surjective, we have 
\[
(a \otimes_{\Lambda}S) \End_{\Gamma}(P_1 \otimes_{\Lambda}S)+\End_{\Gamma}(P_0 \otimes_{\Lambda}S) (a\otimes_{\Lambda}S) = 
\Hom_{\Gamma}(P_1 \otimes_{\Lambda}S,P_0 \otimes_{\Lambda}S).
\]
This implies that any morphism 
\[
\xymatrix{
\cdots \ar[r] & 0 \ar[r] \ar[d] & 0 \ar[r] \ar[d] & P_1 \otimes_{\Lambda}S \ar[r]^{a\otimes_{\Lambda}S} \ar[d] & P_0 \otimes_{\Lambda}S \ar[r] \ar[d] & 0  \ar[r] \ar[d]& \cdots  \\
\cdots \ar[r] & 0 \ar[r] &P_1 \otimes_{\Lambda}S \ar[r]_{a\otimes_{\Lambda}S} & P_0 \otimes_{\Lambda}S \ar[r] & 0 \ar[r] & 0\ar[r] & \cdots 
}
\]
of complexes of $\Gamma$-modules is null-homotopic. 
Since $\Gamma$ is a division algebra,
any $f \in \Hom_{\Gamma}(\Ker (a \otimes_{\Lambda}S), \Coker (a \otimes_{\Lambda}S))$ can be extended to $\tilde{f} \in \Hom_{\Gamma}(P_1 \otimes_{\Lambda}S,P_0 \otimes_{\Lambda}S)$.
But 
\[
\xymatrix{
\cdots \ar[r] & 0 \ar[r] \ar[d] & 0 \ar[r] \ar[d] & P_1 \otimes_{\Lambda}S \ar[r]^{a\otimes_{\Lambda}S} \ar[d]^{\tilde{f}} & P_0 \otimes_{\Lambda}S \ar[r] \ar[d] & 0  \ar[r] \ar[d]& \cdots  \\
\cdots \ar[r] & 0 \ar[r] & P_1 \otimes_{\Lambda}S \ar[r]_{a\otimes_{\Lambda}S} & P_0 \otimes_{\Lambda}S \ar[r] & 0 \ar[r] & 0\ar[r] & \cdots 
}
\]
is null-homotopic, we have $f=0$. Thus we have $\Hom_{\Gamma}(\Ker (a \otimes_{\Lambda}S), \Coker (a \otimes_{\Lambda}S))=0$. 
Since $\Gamma$ is a division algebra, we have $\Tor^{\Lambda}_1(T,S)=\Ker (a \otimes_{\Lambda}S)=0$ or $T \otimes_{\Lambda} S = \Coker (a \otimes_{\Lambda}S)=0$.
\end{proof}

\begin{proof}[Proof of Theorem \ref{classi-tilt}]
First we show that $I_i$ is a tilting $\Lambda$-module for any $i$. 
We remark that $I_i=(\bigoplus_{j \neq i}e_j\Lambda) \oplus e_iI_i$. 
By Proposition \ref{pp-Koszul}, there exists an exact sequence
\[
0 \longrightarrow e_i\Lambda \longrightarrow \bigoplus_{j \rightarrow i}e_j\Lambda \longrightarrow e_iI_i \longrightarrow 0.
\]
Thus we have exact sequences of the forms \eqref{exact-fgp} and \eqref{exact-tilt} in Definition \ref{df_tilt}.
We show $\Ext^1_{\Lambda}(I_i,I_i)=0$. 
By applying $\Hom_{\Lambda}(-,I_i)$ to the exact sequence \eqref{exact1}, we have an exact sequece
\[
0= \Ext^1_{\Lambda}(\Lambda,I_i) \longrightarrow \Ext^1_{\Lambda}(I_i,I_i) \longrightarrow \Ext^2_{\Lambda}(S_i,I_i) \longrightarrow 
\Ext^2_{\Lambda}(\Lambda,I_i) =0.
\]
By this and Lemma \ref{lem:1-1}, we have
\[
\Ext^1_{\Lambda}(I_i,I_i) \simeq \Ext^2_{\Lambda}(S_i,I_i) \simeq D \Hom_{\Lambda}(I_i,S_i)=0.
\]
Thus $I_i$ is a tilting $\Lambda$-module.

Next we show $\End_{\Lambda}(I_i) \simeq \Lambda$.
By applying $\Hom_{\Lambda}(-,\Lambda)$ to the exact sequence \eqref{exact1}, we have an exact sequece
\[
0 \longrightarrow \Hom_{\Lambda}(S_i,\Lambda) \longrightarrow \Hom_{\Lambda}(\Lambda,\Lambda) 
\longrightarrow \Hom_{\Lambda}(I_i,\Lambda) \longrightarrow \Ext^1_{\Lambda}(S_i,\Lambda).
\]
Since $\Ext^j_{\Lambda}(S_i,\Lambda) \simeq D \Ext^{2-j}_{\Lambda}(\Lambda,S_i)=0$ for $j=0,1$ by Lemma \ref{lem:1-1}, 
we have $\Hom_{\Lambda}(\Lambda,\Lambda) \simeq \Hom_{\Lambda}(I_i,\Lambda)$.
By applying $\Hom_{\Lambda}(I_i,-)$ to the exact sequence \eqref{exact1}, we have an exact sequence
\[
0 \longrightarrow \Hom_{\Lambda}(I_i,I_i) \longrightarrow \Hom_{\Lambda}(I_i,\Lambda) \longrightarrow  \Hom_{\Lambda}(I_i,S_i)=0.
\]
Thus we have
\[
\Hom_{\Lambda}(\Lambda,\Lambda) \simeq \Hom_{\Lambda}(I_i,\Lambda) \simeq \Hom_{\Lambda}(I_i,I_i).
\]
Since the above isomorphism is given by 
\[
\Hom_{\Lambda}(\Lambda,\Lambda) \ni a \cdot \longmapsto a \cdot \in \Hom_{\Lambda}(I_i,I_i),
\]
we have a $K$-algebra isomorphism $\Lambda \simeq \End_{\Lambda}(I_i)$.

Finally we show that $I_{i_{\ell}} \cdots I_{i_2}I_{i_1}$ is a tilting 
$\Lambda$-module with $\End_{\Lambda}(I_{i_{\ell}} \cdots I_{i_2}I_{i_1})=\Lambda$
for any $\ell \in \mathbb{N}\cup\{0\}$ and $i_1,i_2,\ldots,i_{\ell} \in Q_0$ by induction on $\ell$.
If $I_{i_{\ell}} \cdots I_{i_2}I_{i_1}=I_{i_{\ell}} \cdots I_{i_2}$, it is a tilting $\Lambda$-module with 
$\End_{\Lambda}(I_{i_{\ell}} \cdots I_{i_2}I_{i_1})=\Lambda$ by inductive hypothesis. 

We assume $I_{i_{\ell}} \cdots I_{i_2}I_{i_1} \neq I_{i_{\ell}} \cdots I_{i_2}$.
By \cite[Corollary 1.7.(3)]{Yak}, $I_{i_{\ell}} \cdots I_{i_2} \Ltensor_{\Lambda} I_{i_1}$ is a tilting complex in $\dcat{}{\mathrm{Mod} \Lambda}$,  
so we have 
\[
\Hom_{\dcat{}{\mathrm{Mod}\Lambda}}(I_{i_{\ell}} \cdots I_{i_2} \Ltensor_{\Lambda} I_{i_1},I_{i_{\ell}} \cdots I_{i_2} \Ltensor_{\Lambda} I_{i_1}) \simeq 
\Hom_{\Lambda}(I_{i_{\ell}} \cdots I_{i_2},I_{i_{\ell}} \cdots I_{i_2}) \simeq \Lambda.
\]
Hence it is enough to show that
\[
(I_{i_{\ell}} \cdots I_{i_2}) \Ltensor_{\Lambda} I_{i_1}= (I_{i_{\ell}} \cdots I_{i_2}) \otimes_{\Lambda} I_{i_1}=
I_{i_{\ell}} \cdots I_{i_2} I_{i_1}
\]
and $\mathrm{pd} (I_{i_{\ell}} \cdots I_{i_2} I_{i_1}) \leq 1$.

Since $\mathrm{pd}S_{i_1}=2$, we have $\Tor^{\Lambda}_{j}(I_{i_{\ell}} \cdots I_{i_2},I_{i_1}) \simeq 
\Tor^{\Lambda}_{j+2}(\Lambda/(I_{i_{\ell}} \cdots I_{i_2}),S_{i_1})=0$ for any $j \neq 0$. 
Thus we have $(I_{i_{\ell}} \cdots I_{i_2}) \Ltensor_{\Lambda} I_{i_1}= (I_{i_{\ell}} \cdots I_{i_2}) \otimes_{\Lambda} I_{i_1}$.

Now we show $\Tor^{\Lambda}_1(I_{i_{\ell}} \cdots I_{i_2},S_{i_1})=0$. 
By applying $I_{i_{\ell}} \cdots I_{i_2} \otimes_{\Lambda}-$ to an exact sequence
\[
0 \longrightarrow I_{i_1} \longrightarrow \Lambda \longrightarrow S_{i_1} \longrightarrow 0,
\]
we have a commutative diagram
\[
\xymatrix{
\Tor^{\Lambda}_1(I_{i_{\ell}} \cdots I_{i_2},S_{i_1}) \ar[r] & (I_{i_{\ell}} \cdots I_{i_2}) \otimes_{\Lambda} I_{i_1} \ar[r] \ar[d]^f & (I_{i_{\ell}} \cdots I_{i_2}) \otimes_{\Lambda} \Lambda \ar[d]^g \ar[r] &  
 (I_{i_{\ell}} \cdots I_{i_2}) \otimes_{\Lambda} S_{i_1}\\
& I_{i_{\ell}} \cdots I_{i_2} I_{i_1} \ar[r] &  I_{i_{\ell}} \cdots I_{i_2} &
}
\]
such that the first row is exact, $f$ is an epimorphism and $g$ is an isomorphism. 
If $I_{i_{\ell}} \cdots I_{i_2} \otimes_{\Lambda} S_{i_1}=0$, we have $I_{i_{\ell}} \cdots I_{i_2} I_{i_1} =  I_{i_{\ell}} \cdots I_{i_2}$. This is a contradiction, hence we have $I_{i_{\ell}} \cdots I_{i_2} \otimes_{\Lambda} S_{i_1} \neq 0$. 
By Lemma \ref{IRprivate} and the induction hypothesis, $\Tor^{\Lambda}_1(I_{i_{\ell}} \cdots I_{i_2},S_{i_1})=0$ holds.
Thus $f$ is a monomorphism, hence we have $(I_{i_{\ell}} \cdots I_{i_2}) \otimes_{\Lambda} I_{i_1}=I_{i_{\ell}} \cdots I_{i_2} I_{i_1}$.
Since $\mathrm{pd}(I_{i_{\ell}} \cdots I_{i_2}) \leq 1$ and $\mathrm{pd}((I_{i_{\ell}} \cdots I_{i_2})/(I_{i_{\ell}} \cdots I_{i_2}I_{i_1})) \leq 2$, 
we have $\mathrm{pd}(I_{i_{\ell}} \cdots I_{i_2}I_{i_1}) \leq 1$.
The assertion follows.
\end{proof}

\subsection{Description of $\mathcal{I}(\Lambda)$ via the Coxeter group}
\label{Coxeter}

In the previous subsection we constructed the set $\mathcal{I}(\Lambda)$ of tilting $\Lambda$-modules. 
However elements in $\mathcal{I}(\Lambda)$ has many expressions (e.g.~$I_i=I_i^2$).
Thus, in this subsection, we describe elements in $\mathcal{I}(\Lambda)$ by using the Coxeter group associated to $Q$, which was investigated in \cite{BIRS}. 
We keep the notations in the previous subsection.

First we recall the definition of the Coxeter group associated to a finite quiver.

\begin{df}\label{def-Weyl}
For any finite connected quiver $Q$ with no loops (not necessarily non-Dynkin), 
the \emph{Coxeter group} $W_Q$ associated to $Q$ is defined as a group whose generators are $s_0,\ldots,s_n$ with the relations
\[
\begin{cases}
s^2_i=1, \\
\mbox{$s_is_j=s_js_i$ if there is no arrows between $i$ and $j$ in $Q$,} \\
\mbox{$s_is_js_i=s_js_is_j$ if there is precisely one arrow between $i$ and $j$ in $Q$.}
\end{cases}
\]
In particular, if $Q$ is a Dynkin quiver, we call $W_Q$ the \emph{(finite) Weyl group}, 
and if $Q$ is an extended Dynkin quiver, we call $W_Q$ the \emph{affine Weyl group}.

Define the \emph{length} $\ell(w)$ of $w$ to be the smallest $r$ for which such an expression exists and call the expression \emph{reduced}.
By convention $\ell(1)=0$.
Clearly $\ell(w)=1$ if and only if $w = s_i$ for some $i \in Q_0$.
\end{df}

Now we define a correspondence between $W_Q$ and $\mathcal{I}(\Lambda)$. 
For $w \in W_Q$, we take a reduced expression $w=s_{i_{\ell}} \cdots s_{i_2}s_{i_1}$. Then we put
\[
I_w := I_{i_{\ell}} \cdots I_{i_2}I_{i_1} \in \mathcal{I}(\Lambda). 
\]
This gives a correspondence
\[
W_Q \ni w \longmapsto I_w \in \mathcal{I}(\Lambda).
\]
The following results imply that the correspondence is actually well-defined and bijective. We omit proofs because they are shown by the quite same arguments as in \cite{BIRS}.

\begin{prop}[\cite{BIRS} Proposition III.1.8.]\label{rel-tilt}
The following hold.
\begin{enumerate}
\def\labelenumi{(\theenumi)}
\item $I^2_i=I_i$.
\item $I_iI_j=I_jI_i$ if there is no arrows between $i$ and $j$ in $Q$.
\item $I_iI_jI_i=I_jI_iI_j$ if there is precisely one arrow between $i$ and $j$ in $Q$.
\end{enumerate}
\end{prop}

\begin{thm}[\cite{BIRS} Theorem III.1.9.]\label{Weyl-tilt}
The correspondence $W_Q \ni w \longmapsto I_w \in \mathcal{I}(\Lambda)$ is a bijection.
\end{thm}

For any $w \in W_Q$, $\RHom_{\Lambda}(I_w,-)$ and $-\Ltensor_{\Lambda}I_w$ are decomposed by using reduced expression of $w$ as follows.

\begin{prop}\label{red-tilt}
Let $w$ be an element of $W_Q$. We take a reduced expression $w=s_{i_{\ell}} \cdots s_{i_2}s_{i_1}$. 
Then the following hold.
\begin{enumerate}
\def\labelenumi{(\theenumi)} 
\item $\RHom_{\Lambda}(I_w,-)=\RHom_{\Lambda}(I_{i_{\ell}},-) \circ \cdots \circ \RHom_{\Lambda}(I_{i_2},-) \circ \RHom_{\Lambda}(I_{i_1},-)$.
\item $-\Ltensor_{\Lambda}I_w=-\Ltensor_{\Lambda}I_{i_{\ell}}  \Ltensor_{\Lambda} \cdots \Ltensor_{\Lambda}I_{i_2} \Ltensor_{\Lambda}I_{i_1}$.
\end{enumerate}
\end{prop}
\begin{proof}
(2) Since $w=s_{i_{\ell}} \cdots s_{i_2}s_{i_1}$ is a reduced expression and \cite[Proposition III.1.10.]{BIRS}, we have a strict descending chain of tilting ideals
\[
I_{i_{\ell}} \cdots I_{i_2} I_{i_1} \subset I_{i_{\ell}} \cdots I_{i_2} \subset \cdots\cdots \subset I_{i_{\ell}}.
\]
Thus by the proof of Theorem \ref{classi-tilt}, we have
\[
I_{i_{\ell}} \Ltensor_{\Lambda} \cdots \Ltensor_{\Lambda} I_{i_2} \Ltensor_{\Lambda} I_{i_1}=I_{i_{\ell}} \cdots I_{i_2} I_{i_1}=I_w.
\]
The assertion follows. 

(1) There is an adjoint isomorphism
\[
\RHom_{\Lambda}(L\Ltensor_{\Lambda}M,N) \simeq \RHom_{\Lambda}(L,\RHom_{\Lambda}(M,N))
\]
for any $L,N \in \mathcal{D}(\mathrm{Mod}\Lambda)$ and $M \in \mathcal{D}(\mathrm{Mod}(\Lambda^{\op} \otimes_K \Lambda))$.
By (2) and the above isomorphism, we have the assertion.
\end{proof}

\subsection{The change of dimension vectors}
\label{dimension-tilting}

In this subsection, we study equivalences between some full subcategories of $\dcat{}{\mathrm{Mod}\Lambda}$ which are restrictions of auto-equivalences on $\dcat{}{\mathrm{Mod}\Lambda}$ induced by tilting $\Lambda$-modules in $\mathcal{I}(\Lambda)$.
We keep the notations in the previous subsection.

We denote by $\fd \Lambda$ the full category of $\mathrm{Mod}\Lambda$ whose objects consist of finite dimensional $\Lambda$-modules.
The category $\fd \Lambda$ has a duality
\[
D: \fd \Lambda \longrightarrow \fd \Lambda^{\mathrm{op}}
\]
such that $D \circ D$ is isomorphic to the identity functor.

We denote by $\mathcal{D}$ the full subcategory of 
$\mathcal{D}(\mathrm{Mod}\Lambda)$ whose objects consist of complexes whose total homology is in $\fd \Lambda$.

First we observe that tilting modules which lie in $\mathcal{I}(\Lambda)$ induce triangle auto-equivalences on $\mathcal{D}$.

\begin{lem}\label{D^b(nilp)-eq}
Let $T$ be a tilting $\Lambda$-module which lies in $\mathcal{I}(\Lambda)$.
Then the restriction of the triangle equivalence
\[
\xymatrix{
\dcat{}{\mathrm{Mod}\Lambda} \ar@<0.5ex>[rr]^{\RHom_{\Lambda}(T,-)} & & \dcat{}{\mathrm{Mod}\Lambda} \ar@<0.5ex>[ll]^{-\Ltensor_{\Lambda} T}
}
\]
to $\mathcal{D}$ induces a triangle equivalence
\[
\xymatrix{
\mathcal{D} \ar@<0.5ex>[rr]^{\RHom_{\Lambda}(T,-)} & & \mathcal{D} \ar@<0.5ex>[ll]^{-\Ltensor_{\Lambda} T}
}.
\]
\end{lem}
\begin{proof}
By the condition (1) in the Definition \ref{df_tilt}, for any finite dimensional $\Lambda$-module $M$, 
$\Hom_{\Lambda}(T,M)$ and $\Ext^1_{\Lambda}(T,M)$ is finite dimensional.
\end{proof}

Next we study the above equivalence in the level of dimension vectors.
We consider a map
\[
[-]:\mathcal{D} \longrightarrow \mathbb{Z}^{\overline{Q}_0}
\]
defined by
\[
[X^{\bullet}] = \sum_{i \in \mathbb Z}(-1)^i \udim \mathrm{H}^i(X^{\bullet}).
\]

We define an action of $W_Q$ on $\mathbb{Z}^{\overline{Q}_0}$ by
\[
s_i(\mathbf{x}) := \mathbf{x} - (\mathbf{x},\mathbf{e}_i) \mathbf{e}_i.
\]
Then the following result holds.

\begin{thm}\label{tilt-Gro}
The following diagrams commute.
\[
\xymatrix{
\mathcal{D} \ar[rr]^{\RHom_{\Lambda}(I_i,-)} \ar[d]_{[-]} & & \mathcal{D} \ar[d]^{[-]}\\
\mathbb{Z}^{\overline{Q}_0} \ar[rr]^{s_i} & & \mathbb{Z}^{\overline{Q}_0} 
}\quad
\xymatrix{
\mathcal{D} \ar[rr]^{-\Ltensor_{\Lambda} I_i} \ar[d]_{[-]} & & \mathcal{D} \ar[d]^{[-]}\\
\mathbb{Z}^{\overline{Q}_0} \ar[rr]^{s_i} & & \mathbb{Z}^{\overline{Q}_0}
}
\]
\end{thm}

\begin{proof}
It is enough to show that for any finite dimensional $\Lambda$-module $M$, $s_i[M]=[\RHom_{\Lambda}(I_i,M)]$ holds.
Let $M$ be a finite dimensional $\Lambda$-module, and put $\alpha=(\alpha_i):=\udim M$. 

First we have
\[
\Hom_{\Lambda}(I_i,M)e_j = \begin{cases}
\Hom_{\Lambda}(e_iI_i,M) & (i=j) \\
\Hom_{\Lambda}(e_j \Lambda,M) \simeq Me_j & (i \neq j)
\end{cases}
\]
and
\[
\Ext^1_{\Lambda}(I_i,M)e_j = \begin{cases}
\Ext^1_{\Lambda}(e_iI_i,M) & (i=j) \\
\Ext^1_{\Lambda}(e_j \Lambda,M) =0 & (i \neq j).
\end{cases}
\]

Next by applying $\Hom_{\Lambda}(-,M)$ to the exact seqneuce
\[
0 \longrightarrow e_i\Lambda \longrightarrow \bigoplus_{j \rightarrow i}e_j\Lambda \longrightarrow e_iI_i \longrightarrow 0
\]
obtained from Proposition \ref{pp-Koszul}, we have an exact sequence
\[
0 \longrightarrow \Hom_{\Lambda}(e_iI_i,M) \longrightarrow \bigoplus_{j \rightarrow i}\Hom_{\Lambda}(e_j\Lambda,M) \longrightarrow \Hom_{\Lambda}(e_i\Lambda,M) \longrightarrow \Ext^1_{\Lambda}(e_iI_i,M) \longrightarrow 0.
\]
Thus we have 
\begin{eqnarray*}
\dim \Hom_{\Lambda}(e_iI_i,M)-\dim \Ext^1_{\Lambda}(e_iI_i,M) &=& \sum_{j \rightarrow i} \dim \Hom_{\Lambda}(e_j\Lambda,M)-\dim \Hom_{\Lambda}(e_i\Lambda,M) \\
&=& \sum_{ta=i}\alpha_{sa}-\alpha_i.
\end{eqnarray*}
Consequently we have
\begin{eqnarray*}
[\RHom_{\Lambda}(I_i,M)] &=& \udim \Hom_{\Lambda}(I_i,M)-\udim \Ext^1_{\Lambda}(I_i,M) \\
&=& \sum_{j \neq i} \alpha_i \mathbf{e}_j +(\dim \Hom_{\Lambda}(e_iI_i,M)-\dim \Ext^1_{\Lambda}(e_iI_i,M))\mathbf{e}_i \\
&=& \sum_{j \neq i} \alpha_j \mathbf{e}_j +\left(\sum_{ta=i}\alpha_{sa}-\alpha_i\right)\mathbf{e}_i \\
&=& \alpha - \left(2 \alpha_i - \sum_{ta=i}\alpha_{sa} \right)\mathbf{e}_i \\
&=& \alpha - (\alpha,\mathbf{e}_i)\mathbf{e}_i = s_i \alpha.
\end{eqnarray*}
On the other hand, by $s_i^2=1$ and $[\RHom_{\Lambda}(I_i,M\Ltensor_{\Lambda}I_i)] = [M]$ for any $M \in \mathcal D$, we have $[M\Ltensor_{\Lambda}I_i)] = s_i[M]$.  
\end{proof}

\subsection{Torsion pairs}
\label{torsion-tilting}

In the latter sections, we study semistable modules. 
To study a behavior of them in more detail, it is better to restrict the derived equivalence induced by a tilting module to the modules categories.
At that time, the notion of the torsion pair becomes important. 
We define full subcategories $\mathcal{T}(T)$ and $\mathcal{F}(T)$ of $\fd \Lambda$ by
\[
\mathcal{T}(T):=\{ M \in \fd \Lambda \ | \ \ \Ext^1_{\Lambda}(T,M)=0 \}
\]
and
\[
\mathcal{F}(T):=\{ M \in \fd \Lambda \ | \ \Hom_{\Lambda}(T,M)=0 \}.
\]
We also define full subcategories $\mathcal{X}(T)$ and $\mathcal{Y}(T)$ of $\fd \Lambda$ by
\[
\mathcal{X}(T):=\{ M \in \fd \Lambda \ | \ M \otimes_{\Lambda} T =0 \}
\]
and
\[
\mathcal{Y}(T):=\{ M \in \fd \Lambda \ | \ \Tor^{\Lambda}_1(M,T)=0 \}.
\]
The pairs $(\mathcal T(T),\mathcal Y(T))$ and $(\mathcal F(T),\mathcal X(T))$ form torsion pairs in $\fd \Lambda$.
By the definition, $\mathcal{T}(T)$ and $\mathcal{X}(T)$ are closed under images, extensions and finite direct sums,
and $\mathcal{F}(T)$ and $\mathcal{Y}(T)$ are closed under submodules, extensions and finite direct sums.

The following result immediately follows from Lemma \ref{D^b(nilp)-eq}.

\begin{lem}\label{Brenner-Butler}
Let $T$ be a tilting $\Lambda$-module which lies in $\mathcal{I}(\Lambda)$.
Then there are categorical equivalences
\[
\xymatrix{
\mathcal{T}(T) \ar@<0.5ex>[rr]^{\Hom_{\Lambda}(T,-)} & & \mathcal{Y}(T) \ar@<0.5ex>[ll]^{- \otimes_{\Lambda} T}
}
\]
and
\[
\xymatrix{
\mathcal{F}(T) \ar@<0.5ex>[rr]^{\Ext^1_{\Lambda}(T,-)} & & \mathcal{X}(T) \ar@<0.5ex>[ll]^{\Tor^{\Lambda}_1(-,T)}
}.
\]
\end{lem}

In the case $T=I_i$, we have explicit descriptions of the above full subcategories.

\begin{lem}\label{tor-I_i}
The following hold.
\begin{enumerate}
\def\labelenumi{(\theenumi)}
\item $\mathcal T(I_i) = \{ M \in \fd \Lambda \mid S_i \text{\ is not a direct summand of\ } M/(MI) \}$.
\item $\mathcal F(I_i) = \add S_i$.
\item $\mathcal X(I_i) = \add S_i$
\item $\mathcal Y(I_i) = \{ M \in \fd \Lambda \mid S_i \text{\ is not a direct summand of\ } \Soc(M) \}$.
\end{enumerate}
\end{lem}
\begin{proof}
(1) Let $M$ be a finite dimensional $\Lambda$-module. 
By applying $\Hom_{\Lambda}(-,M)$ to the exact sequence \eqref{exact1},  we have $\Ext^1_{\Lambda}(I_i,M) \simeq \Ext^2_{\Lambda}(S_i,M)$, 
and by Lemma \ref{lem:1-1}, we have $\Ext^1_{\Lambda}(I_i,M) \simeq D\Hom_{\Lambda}(M,S_i)$. Thus the assertion follows. 

(2) Let $M$ be a finite dimensional $\Lambda$-module. 
If $M$ lies in $\add S_i$, we have $\Hom_{\Lambda}(I_i,M)=0$. 
Conversely we assume $\Hom_{\Lambda}(I_i,M)=0$. 
Then by applying $\Hom_{\Lambda}(-,M)$ to the exact sequence \eqref{exact1},  we have $M \simeq \Hom_{\Lambda}(S_i,M) \in \add S_i$.

(3) and (4) follow from similar arguments.
\end{proof}

The following immediately follows from Theorem \ref{tilt-Gro}.

\begin{cor}\label{change-dim}
For a finite dimensional $\Lambda$-module $M$, the following hold.
\begin{enumerate}
\def\labelenumi{(\theenumi)}
\item If $M \in \mathcal T(I_i)$, then $\udim \Hom_{\Lambda}(I_i,M) = [\RHom_{\Lambda}(I_i,M)] = s_i(\udim M)$.
\item If $M \in \mathcal F(I_i)$, then $\udim \Ext^1_{\Lambda}(I_i,M) = -[\RHom_{\Lambda}(I_i,M)] = -s_i(\udim M)$.
\item If $M \in \mathcal X(I_i)$, then $\udim \Tor^{\Lambda}_1(M,I_i) = -[M \Ltensor_{\Lambda} I_i] = - s_i(\udim M)$.
\item If $M \in \mathcal Y(I_i)$, then $\udim M \otimes_{\Lambda} I_i = [M \Ltensor_{\Lambda} I_i] = s_i(\udim M)$.
\end{enumerate}
\end{cor}

\subsection{The categories of finite dimensional modules over preprojective algebras of extended Dynkin quivers}
\label{fdDynkin}

In this subsection, we determine the structure of the category $\fd \Lambda$ of finite dimensional modules over the preprojective algebra $\Lambda$ of an extended Dynkin quiver $Q$ whose double $\overline{Q}$ is one of quivers appeared in Figure \ref{McKay}.

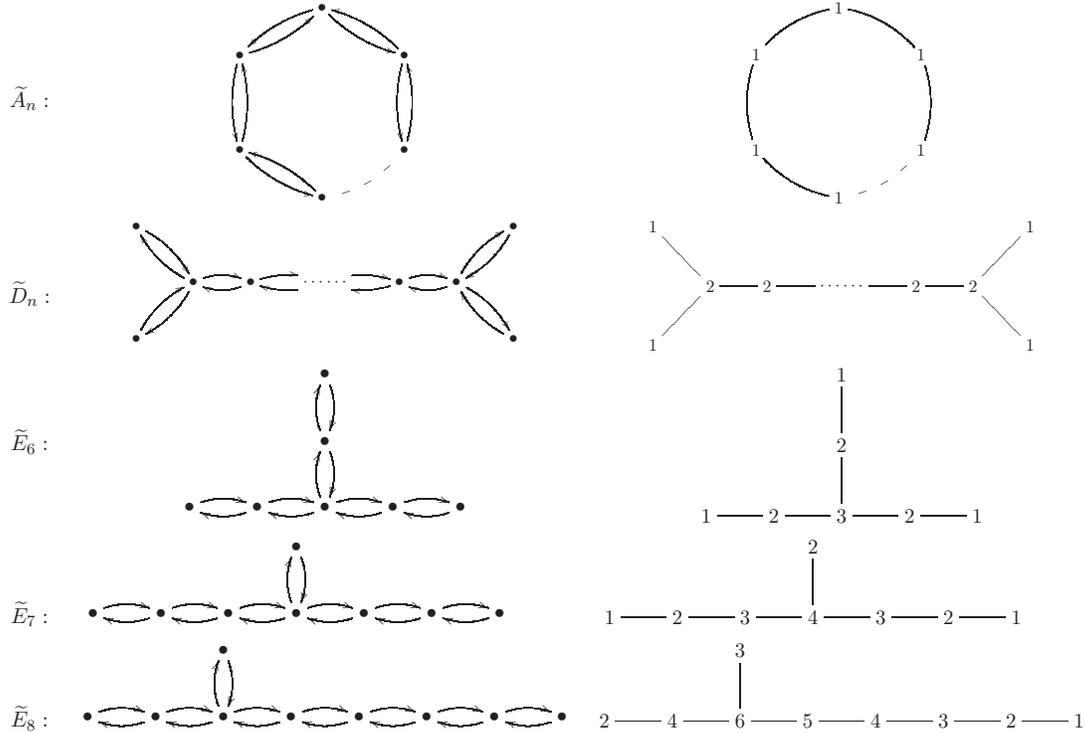
\begin{figure}[htbp]
\begin{center}
\scalebox{.75}{
$\begin{array}{ccc}
\widetilde{A}_n : &
\scalebox{.8}{$
\def\alphanum{\ifcase\xypolynode\or \bullet\or\bullet\or\bullet\or\bullet\or\bullet\or\bullet\fi}
\xy \xygraph{!{/r5pc/:} [] !P6"A"{~>{} ~*{\alphanum} ~={90}}}
\ar@/_.4pc/ "A1";"A2"
\ar@/_.4pc/ "A2";"A3"
\ar@/_.4pc/ "A3";"A4"
\ar@/_/@{--} "A4";"A5"
\ar@/_.4pc/ "A5";"A6"
\ar@/_.4pc/ "A6";"A1"
\ar@/_.4pc/ "A2";"A1"
\ar@/_.4pc/ "A3";"A2"
\ar@/_.4pc/ "A4";"A3"
\ar@/_.4pc/ "A6";"A5"
\ar@/_.4pc/ "A1";"A6"
\endxy$
}
&
\scalebox{.8}{$
\def\alphanum{\ifcase\xypolynode\or{\huge1}\or{\huge1}\or{\huge1}\or{\huge1}\or{\huge1}\or{\huge1}\fi}
\xy \xygraph{!{/r5pc/:} [] !P6"A"{~>{} ~*{\alphanum} ~={90}}}
\ar@/_/@{-} "A1";"A2"
\ar@/_/@{-} "A2";"A3"
\ar@/_/@{-} "A3";"A4"
\ar@/_/@{--} "A4";"A5"
\ar@/_/@{-} "A5";"A6"
\ar@/_/@{-} "A6";"A1"
\endxy$
}
\\
\xymatrix{ \\  \widetilde{D}_n : } & 
\scalebox{.8}{\xymatrix{
\bullet \ar@/^.4pc/[rd] \ar@{<-}@/_.4pc/[rd] & & & & & & \bullet \ar@/^.4pc/[ld] \ar@{<-}@/_.4pc/[ld]\\
 & \bullet \ar@/^.4pc/[r] \ar@{<-}@/_.4pc/[r] & \bullet \ar@/^.4pc/[r] \ar@{<-}@/_.4pc/[r] & \cdots\cdots \ar@/^.4pc/[r] \ar@{<-}@/_.4pc/[r] & \bullet \ar@/^.4pc/[r] \ar@{<-}@/_.4pc/[r] & \bullet & \\
\bullet \ar@/^.4pc/[ru] \ar@{<-}@/_.4pc/[ru] & & & & & & \bullet \ar@/^.4pc/[lu] \ar@{<-}@/_.4pc/[lu]
}}
&
\scalebox{.8}{\xymatrix{
{\LARGE1} \ar@{-}[rd] &&&&&& {\LARGE1} \ar@{-}[ld]\\
& {\LARGE2} \ar@{-}[r] & {\LARGE2} \ar@{-}[r] & \cdots\cdots \ar@{-}[r] & {\LARGE2} \ar@{-}[r] & {\LARGE2} & \\
{\LARGE1} \ar@{-}[ru] &&&&&& {\LARGE1} \ar@{-}[lu]
}}
\\
\xymatrix{ \\  \widetilde{E}_6 : } &  
\scalebox{.95}{\xymatrix{
& & \bullet \ar@/^.4pc/[d] \ar@{<-}@/_.4pc/[d] & & \\
& & \bullet \ar@/^.4pc/[d] \ar@{<-}@/_.4pc/[d] & & \\
\bullet \ar@/^.4pc/[r] \ar@{<-}@/_.4pc/[r] & \bullet \ar@/^.4pc/[r] \ar@{<-}@/_.4pc/[r] & \bullet \ar@/^.4pc/[r] \ar@{<-}@/_.4pc/[r] & \bullet \ar@/^.4pc/[r] \ar@{<-}@/_.4pc/[r] & \bullet 
}}
&
\scalebox{.95}{\xymatrix{
&& 1 \ar@{-}[d] && \\
&& 2 \ar@{-}[d] && \\
1 \ar@{-}[r] & 2 \ar@{-}[r] & 3 \ar@{-}[r] & 2 \ar@{-}[r] & 1 
}}
\\
\xymatrix{ \\  \widetilde{E}_7 : } &
\scalebox{.95}{\xymatrix{
 & & & \bullet \ar@/^.4pc/[d] \ar@{<-}@/_.4pc/[d] & & & \\
\bullet \ar@/^.4pc/[r] \ar@{<-}@/_.4pc/[r] & \bullet \ar@/^.4pc/[r] \ar@{<-}@/_.4pc/[r]  & \bullet \ar@/^.4pc/[r] \ar@{<-}@/_.4pc/[r]  &\bullet \ar@/^.4pc/[r] \ar@{<-}@/_.4pc/[r] &\bullet \ar@/^.4pc/[r] \ar@{<-}@/_.4pc/[r] & \bullet \ar@/^.4pc/[r] \ar@{<-}@/_.4pc/[r] & \bullet &
}}
&
\scalebox{.95}{\xymatrix{
&&& 2 \ar@{-}[d] &&& \\
1 \ar@{-}[r] & 2 \ar@{-}[r]  & 3 \ar@{-}[r]  & 4 \ar@{-}[r] & 3 \ar@{-}[r] & 2 \ar@{-}[r] & 1 &
}}
\\  
\xymatrix{ \\  \widetilde{E}_8 : } &  
\scalebox{.95}{\xymatrix{
 & &  \bullet \ar@/^.4pc/[d] \ar@{<-}@/_.4pc/[d] & & & & & \\
\bullet \ar@/^.4pc/[r] \ar@{<-}@/_.4pc/[r] &\bullet \ar@/^.4pc/[r] \ar@{<-}@/_.4pc/[r] & \bullet \ar@/^.4pc/[r] \ar@{<-}@/_.4pc/[r] &\bullet \ar@/^.4pc/[r] \ar@{<-}@/_.4pc/[r] &\bullet \ar@/^.4pc/[r] \ar@{<-}@/_.4pc/[r] & \bullet \ar@/^.4pc/[r] \ar@{<-}@/_.4pc/[r] & \bullet \ar@/^.4pc/[r] \ar@{<-}@/_.4pc/[r] &\bullet 
}}
&
\scalebox{.95}{\xymatrix{
&&  3 \ar@{-}[d] &&&&& \\
2 \ar@{-}[r] & 4 \ar@{-}[r] & 6 \ar@{-}[r] & 5 \ar@{-}[r] & 4 \ar@{-}[r] & 3 \ar@{-}[r] & 2 \ar@{-}[r] & 1 
}}
\end{array}$
}
\caption{The double $\overline{Q}$ of extended Dynkin quivers and dimension vectors $\mathbf d$}
\label{McKay}
\end{center}
\end{figure}

These preprojective algebras are important from viewpoint of quotient singularities.
Each quiver $\overline{Q}$ is called the McKay quiver of $G$ which is a finite subgroup of $SL(2,K)$ of type ${A}_n,{D}_n, {E}_6,{E}_7$ or ${E}_8$.
The skew group algebra $S*G$ is Morita equivalent to the preprojective algebra of $Q$ (see \cite{R-VdB,BSW}).
The $\Lambda$-modules of dimension vector $\mathbf d$ are important since they form the moduli spaces which are isomorphic to the minimal resolution of the Kleinian singularity $\mathbb A^2/G$.

The purpose of this subsection is to represent $\fd \Lambda$ as a direct product of its full subcategories.
A $\Lambda$-module $M$ is called \emph{nilpotent} if there exists an $m \in \mathbb{N}$ such that $MI^m=0$.
We remark that a $\Lambda$-module $M$ is finite dimensional nilpotent if and only if $M$ has finite length and its composition factors consist of $S_0,S_1,\ldots,S_n$.
We define the full subcategory $\nilp \Lambda$ of $\fd\Lambda$ which consist of finite dimensional nilpotent $\Lambda$-modules.
We denote by $\mathcal{S}$ the set of isomorphism classes of finite dimensional simple $\Lambda$-modules which are not isomorphic to $S_0,S_1,\ldots,S_n$. 
For any $S \in \mathcal{S}$, we define the full subcategory $\fd_S \Lambda$ of $\fd \Lambda$ which consists of finite dimensional $\Lambda$-modules whose composition factors consist only of $S$.

\begin{prop}\label{cat_fd}
The following holds.
\[
\fd \Lambda= (\prod_{S \in \mathcal{S}}\fd_S \Lambda) \times \nilp \Lambda.
\]
\end{prop}

In the rest of this subsection, we give a proof of Proposition \ref{cat_fd}.
The assertion of Proposition \ref{cat_fd} follows from Lemma \ref{lem2_cat_fd} (2).

We use the symmetric bilinear form $(-,-)$ defined in subsection $2$.

\begin{lem}\label{lem1_cat_fd}
For $\alpha \in \mathbb Z^{\overline{Q}_0}$, a function
\[
(\alpha,-): \mathbb Z^{\overline{Q}_0} \longrightarrow \mathbb{Z}
\]
is zero if and only if $\alpha=m \mathbf{d}$ for some $m \in \mathbb{Z}$ where the $\mathbf{d}$ is the vector appeared in Figure \ref{McKay}.
\end{lem}

\begin{proof}
The assertion is verified by easy calculations.
\end{proof}

\begin{lem}\label{lem2_cat_fd}
Let $M$ be a finite dimensional simple $\Lambda$-module which is not nilpotent. Then the following hold.
\begin{enumerate}
\def\labelenumi{(\theenumi)}
\item The dimension vector of $M$ is $m \mathbf{d}$ for some $m \in \mathbb{N}$.
\item Let $N$ be a simple $\Lambda$-module.
If $M$ is not isomorphic to $N$, then we have $\Ext^1_{\Lambda}(M,N)=\Ext^1_{\Lambda}(N,M)=0$.
\end{enumerate}
\end{lem}
\begin{proof}
(1) We write $\mathbf{d}=(d_0,d_1,\ldots,d_n)$.
Since $M$ is not isomorphic to $S_i$ for any $i \in \overline{Q}_0$, we have
\[
\dim \Ext^1_{\Lambda}(M,S_i) =  - (M,S_i)
\]
for any $i \in \overline{Q}_0$ by Lemma \ref{CB-form}. Thus we have $(M,S_i) \leq 0$ for any $i \in \overline{Q}_0$. 
On the other hand, we have
\[
0=(\mathbf{d},\udim M)=\sum_{i \in \overline{Q}_0} d_i (S_i,M).
\]
Since $d_i (M,S_i) \leq 0$ for any $i \in \overline{Q}_0$, we have $(S_i,M)=0$. 
Consequently a function
\[
(\udim M,-): \mathbb Z^{\overline{Q}_0} \longrightarrow \mathbb{Z}
\]
is zero. By Lemma \ref{lem1_cat_fd}, the assertion follows.

(2) The assertion follows from Lemma \ref{CB-form} and (1).
\end{proof}

As a consequence of  Lemma \ref{lem2_cat_fd}, we have the following result.

\begin{cor}\label{vec->simp}
Let $M$ be an indecomposable finite dimensional $\Lambda$-module. If $\udim M=\mathbf{d}$, $M$ is either nilpotent or simple.
\end{cor}

\section{Moduli space of modules}
\label{moduli}

We recall moduli spaces of modules defined in \cite{Ki}.
In this section, $\Lambda = K\overline{Q}/\langle R \rangle$ denotes the preprojective algebra associated to a finite connected quiver $Q$ with no loops. 
Denote by $Q_0 = \{0,\ldots,n \}$ the vertex set of $Q$. We confuse $\Lambda$-modules and representations of $(Q,R)$ since there is a categorical equivalence between the category of $\Lambda$-modules and the category of representations of $(Q,R)$ (cf.~\cite[Chapter 3]{ASS}).
We regard a dimension vector $(\alpha_i)_{i \in Q_0}$ as a element $\sum_{i \in Q_0} \alpha_i \e_i$ in $\mathbb Z^{Q_0}$ where $\e_0,\ldots,\e_n$ denote the canonical basis of $\mathbb Z^{Q_0}$. 
Also we denote by $(\mathbb Z^{Q_0})^{*}$ the dual lattice of $\mathbb Z^{Q_0}$ with the dual basis $\e_0^*,\ldots,\e_n^*$.
We define the parameter space 
\[
\Theta := (\mathbb Z^{Q_0})^{*}\otimes_{\mathbb Z} \mathbb Q.
\] 
By the canonical pairing, we define $\theta(M) := \langle \theta, \udim M \rangle = \sum_{i \in Q_0}\theta_i \dim_K(Me_i)$ for any $\theta = (\theta_i)_{i \in Q_0} \in \Theta$ and any finite dimensional $KQ$-module $M$.
King provided the notion of stability for modules. 

\begin{df}[\cite{Ki}]
For any $\theta \in \Theta$, a $KQ$-module $M$ is called $\theta$-{\it semistable} (resp.~$\theta$-{\it stable}) if $\theta(M)=0$ and, for any non-zero proper submodule $N$ of $M$, $\theta(N) \geq 0$ (resp.~$>0$). Moreover for a given indivisible vector $\alpha$, $\theta$ is called {\it generic} if all $\theta$-semistable modules of dimension vector $\alpha$ are $\theta$-stable.
\end{df}

For any $\theta \in \Theta$ and any dimension vector $\alpha$, we denote by $\M{\theta}{\Lambda}{\alpha}$ the moduli space of $\theta$-semistable $\Lambda$-modules of dimension vector $\alpha$.
In fact, it is a coarse moduli space parametrizing S-equivalemce classes of $\theta$-semistable $\Lambda$-modules of dimension vector $\alpha$ where two $\theta$-semistable modules are S-equivalent if they have filtrations by $\theta$-stable modules with the same associated graded modules.
For an indivisible vector $\alpha$, if $\theta$ is generic, then $\M{\theta}{\Lambda}{\alpha}$ becomes a fine moduli space. In the case, S-equivalence classes are just isomorphism classes.

Recall that the function $(-,-) : \mathbb Z^{Q_0}\times \mathbb Z^{Q_0} \to \mathbb Z$ denotes the symmetric bilinear form defined in Section \ref{pp}.
We define actions of the Coxeter group $W_Q$ associated to $Q$ on $\mathbb Z^{Q_0}$ and $(\mathbb Z^{Q_0})^{*}$ as follows.
For any simple reflection $s_i$, any $\alpha \in \mathbb Z^{Q_0}$ and any $\theta \in (\mathbb Z^{Q_0})^{*}$, we put
\[
s_i\alpha := \alpha - (\alpha,\e_i)\e_i,
\]
\[
s_i \theta := \theta - \theta_i\sum_{j=1}^n (\e_i,\e_j)\e_j^*.
\]
These determine actions of $W_Q$ on $\mathbb Z^{Q_0}$ and $(\mathbb Z^{Q_0})^*$. Moreover it is extended to $\Theta$ linearly.
Then it is easy to see that $s_i\theta(\alpha)=\theta(s_i\alpha)$ for any $\alpha \in \mathbb Z^{Q_0}$ and $\theta \in \Theta$.

\section{Reflection functor}
\label{reflection}

In this section, $\Lambda$ denotes the preprojective algebra associated to a finite connected quiver $Q$ with no loops. 
Recall that, in Section 1, we defined the full subcategory $\s{\theta}{\Lambda}{}$ of $\Mod \Lambda$ consisting of $\theta$-semistable $\Lambda$-modules and the full subcategory $\s{\theta}{\Lambda}{,\alpha}$ of $\s{\theta}{\Lambda}{}$ consisting of $\theta$-semistable $\Lambda$-modules of dimension vector $\alpha$ if $\s{\theta}{\Lambda}{,\alpha}$ is not empty. The category $\s{\theta}{\Lambda}{}$ is closed under extensions and direct summands.
In this section, we give an equivalence between $\s{\theta}{\Lambda}{}$ and $\s{w\theta}{\Lambda}{}$ for any Coxeter element $w \in W_Q$ by using tilting theory over preprojective algebras studied in Section \ref{tilting}.

\subsection{Simple reflection case $w=s_i$}

First of all we consider the case when $w \in W_Q$ is a simple reflection $s_i$ since the general case is obtained by a composition of the simple case.
The required equivalence is given by the tilting module $I_i$, which gives the derived equivalence 
\[
\RHom_{\Lambda}(I_i,-) : \mathcal D(\Mod \Lambda) \to \mathcal D(\Mod \Lambda).
\]

Recall that $\mathcal T(T)$ and $\mathcal Y(T)$ are full subcategories of $\fd \Lambda$ defined in Section \ref{torsion-tilting}.

\begin{lem}\label{TY}
For any $\theta \in \Theta$ the following hold.
\begin{enumerate}
\item If $\theta_i>0$, then $\s{\theta}{\Lambda}{} \subset \mathcal T(I_i)$.
\item If $\theta_i<0$, then $\s{\theta}{\Lambda}{} \subset \mathcal Y(I_i)$.
\end{enumerate} 
\end{lem}

\begin{proof}
(1) Take any $M \in \s{\theta}{\Lambda}{}$.
By Lemma \ref{tor-I_i} we need to show that $S_i$ is not a direct summand of  $M/MI$. If $S_i$ is a direct summand of $M/MI$, there is an exact sequence $0 \to X \to M \to S_i \to 0$, so $\theta_i=\theta(S_i)=\theta(M)-\theta(X)\leq 0$. This contradicts the assumption.

(2) Take any $M \in \s{\theta}{\Lambda}{}$.  
By Lemma \ref{tor-I_i} we need to show that $S_i$ is not a direct summand of  $\Soc(M)$. If $S_i$ is a direct summand of $\Soc(M)$, there is an exact sequence $0 \to S_i \to M \to X \to 0$, so $\theta_i=\theta(S_i)=\theta(M)-\theta(X)\leq 0$. This contradicts the assumption.
\end{proof}

Now we state a main result in this paper.

\begin{thm}\label{simple-reflection}
For any $\theta \in \Theta$ with $\theta_i>0$, there is a categorical equivalence
\[
\xymatrix{
\s{\theta}{\Lambda}{} \ar@<0.5ex>[rr]^{\Hom_{\Lambda}(I_i,-)} & & \s{s_i \theta}{\Lambda}{} \ar@<0.5ex>[ll]^{- \otimes_{\Lambda} I_i}
}.
\]
Under this equivalence $S$-equivalence classes are preserved and $\theta$-stable modules correspond to $s_i\theta$-stable modules.
In particular it induces the equivalence between $\s{\theta}{\Lambda}{,\alpha}$ and $\s{s_i\theta}{\Lambda}{,s_i\alpha}$.
\end{thm}

\begin{proof}
By Lemma \ref{TY}, we have functors $\Hom_{\Lambda}(I_i,-) : \s{\theta}{\Lambda}{} \to \mathcal Y(I_i)$ and $-\otimes_{\Lambda}I_i : \s{s_i\theta}{\Lambda}{} \to \mathcal T(I_i)$. Since $\Hom_{\Lambda}(I_i,-)$ and $-\otimes_{\Lambda}I_i$ give the equivalence between $\mathcal T(I_i)$ and $\mathcal Y(I_i)$, it is sufficient to show that $\Hom_{\Lambda}(I_i,M) \in \s{s_i\theta}{\Lambda}{}$ for any $M \in \s{\theta}{\Lambda}{}$ and $M \otimes_{\Lambda}I_i \in \s{\theta}{\Lambda}{}$ for any $M \in \s{s_i\theta}{\Lambda}{}$.

Take any $M \in \s{\theta}{\Lambda}{}$.
We show that $M':=\Hom_{\Lambda}(I_i,M)$ is $s_i\theta$-semistable. 
Since $M \in \mathcal T(I_i)$, by Corollary \ref{change-dim} we have
\[
(s_i \theta)(M') = (s_i \theta)(s_i(\udim M))=\theta(M)=0.
\]
Take any non-zero proper submodule $N'$ of $M'$. 
Since $\mathcal Y(I_i)$ is closed under submodules, $N' \in \mathcal Y(I_i)$. By applying $-\otimes_{\Lambda}I_i$ to the exact sequence
\[
0 \longrightarrow N' \longrightarrow M' \longrightarrow M'/N' \longrightarrow 0,
\]
we have an exact sequence
\[
0 \longrightarrow \Tor^{\Lambda}_1(M'/N',I_i) \longrightarrow N' \otimes_{\Lambda} I_i \xrightarrow{\ f \ } M.
\]
Since $X:=\IM f$ is a submodule of $M$, we have $\theta(X) \geq 0$. 
By Lemma \ref{isom1} and the 2-CY property, we have 
$\Tor_{1}^{\Lambda}(M'/N',I_i) \simeq \Tor_{2}^{\Lambda}(M'/N',S_i) \simeq D\Ext_{\Lambda}^2(M'/N',S_i) \simeq \Hom_{\Lambda}(S_i,M'/N') \simeq S_i^m$ for some integer $m$.
Therefore by Corollary \ref{change-dim} we have
\begin{eqnarray*}
(s_i \theta)(N') = (s_i \theta)(s_i(\udim N' \otimes_{\Lambda}I_i)) &=& \theta(N' \otimes_{\Lambda}I_i) \\
&=& \theta(\Tor^{\Lambda}_1(M'/N',I_i))+\theta(X) \\
&=& \theta(S_i^m)+\theta(X) \\
&=& m \theta_i+\theta(X) \geq 0.
\end{eqnarray*}
Thus $M'$ is $s_i \theta$-semistable. 
Furthermore if $M$ is $\theta$-stable, then $\theta(X)>0$, so $(s_i \theta)(N')>0$. Hence $M'$ is $s_i\theta$-stable. 

Conversely we take any $M \in \s{s_i\theta}{\Lambda}{}$. 
We show that $M':=M\otimes_{\Lambda}I_i$ is $\theta$-semistable. 
Since $M \in \mathcal Y(I_i)$, by Corollary \ref{change-dim} we have
\[
\theta(M') = \theta(s_i(\udim M))=(s_i\theta)(M)=0.
\]
Take any non-zero proper submodule $N'$ of $M'$. 
Since $\mathcal T(I_i)$ is closed under images, $M'/N' \in \mathcal T(I_i)$. 
Consider an exact sequence
\[
0 \longrightarrow N' \longrightarrow M' \longrightarrow M'/N' \longrightarrow 0.
\]
By applying $\Hom_{\Lambda}(I_i,-)$ to the above exact sequence, we have an exact sequence
\[
M \stackrel{g}{\to} \Hom_{\Lambda}(I_i,M'/N') \to \Ext^1_{\Lambda}(I_i,N') \to 0
\] 
Since $X:= \IM g$ is a factor module of $M$, we have $s_i\theta(X)\leq 0$. 
By the 2-CY property, we have $\Ext^1_{\Lambda}(I_i,N') \simeq \Ext^2_{\Lambda}(S_i,N') \simeq D\Hom_{\Lambda}(N',S_i) \simeq  S^m_i$ for some integer $m$. 
So by Corollary \ref{change-dim} we have
\begin{eqnarray*}
\theta(N') = -\theta(M'/N') 
&=& -\theta(s_i(\udim \Hom_{\Lambda}(I_i,M'/N')) \\
&=& -(s_i\theta)(\udim \Hom_{\Lambda}(I_i,M'/N') \\
&=& -(s_i\theta)(\Ext^1_{\Lambda}(I_i,N'))-(s_i\theta)(X) \\
&=& -(s_i\theta)(S_i^m)-(s_i\theta)(X) \\
&=& -m\theta_i-(s_i\theta)(X) \geq 0.
\end{eqnarray*}
Hence $M'$ is $\theta$-semistable. 
Furthermore if $M$ is $s_i\theta$-stable, then $(s_i\theta)(X)<0$, so $\theta(N')>0$. Hence $M'$ is $\theta$-stable. 

Now it is trivial that the obtained equivalence preserves $S$-equivalence classes, and induces the equivalence between $\s{\theta}{\Lambda}{,\alpha}$ and $\s{s_i\theta}{\Lambda}{,s_i\alpha}$ by Lemma \ref{TY} and Corollary \ref{change-dim}.
\end{proof}

\begin{df}
The functors $\Hom_{\Lambda}(I_i,-)$ and $-\otimes_{\Lambda}I_i$ in Theorem \ref{simple-reflection} are called simple reflection functors, which are denoted by $\bs_i^{+}$ and $\bs_i^-$ respectively. 
\end{df}

As a corollary we have the following results.

\begin{cor}
For any $\theta \in \Theta$ with $\theta_i \neq 0$, the simple reflection functors $\bs_i^+$ and $\bs_i^-$ induce a bijection between closed points of $\M{\theta}{\Lambda}{\alpha}$ and $\M{s_i\theta}{\Lambda}{s_i\alpha}$
\end{cor}

It is natural to hope that the above bijection is extended to an isomorphism of algebraic varieties. Indeed it is proved in the next section.

\begin{ex}
Let $Q$ be a extended Dynkin quiver of type $\widetilde A_2$. 
Put $\mathbf d = (1,1,1)$.
For the parameter $\theta = (\theta_0,\theta_1,\theta_2)$ with $\theta(\mathbf d)$, $\theta_1=0$ and $\theta_2>0$, the following two representations are contained in $\s{\theta}{\Lambda}{,\mathbf d}$ and which give the same S-equivalent class.
\[
\overline{Q}=
\begin{array}{c}
\def\alphanum{\ifcase\xypolynode \or 0 \or 1 \or 2 \fi}
\xy 
\xygraph{
!{/r3pc/:} 
[] !P3"A"{~>{} ~*{\alphanum} }
}
\ar@/^/|-{} "A1";"A2"
\ar@/^/|-{} "A2";"A3"
\ar@/^/|-{} "A3";"A1"
\ar@/^/|-{} "A1";"A3"
\ar@/^/|-{} "A3";"A2"
\ar@/^/|-{} "A2";"A1"
\endxy \\
\end{array}
\quad M_1=
\begin{array}{c}
\def\alphanum{\ifcase\xypolynode \or K \or K \or K \fi}
\xy 
\xygraph{
!{/r3pc/:} 
[] !P3"A"{~>{} ~*{\alphanum} }
}
\ar@/^/|-{a} "A1";"A2"
\ar@/^/|-{0} "A2";"A3"
\ar@/^/|-{0} "A3";"A1"
\ar@/^/|-{1} "A1";"A3"
\ar@/^/|-{b} "A3";"A2"
\ar@/^/|-{0} "A2";"A1"
\endxy \\
\end{array}
M_2=
\begin{array}{c}
\def\alphanum{\ifcase\xypolynode \or K \or K \or K \fi}
\xy 
\xygraph{
!{/r3pc/:} 
[] !P3"A"{~>{} ~*{\alphanum} }
}
\ar@/^/|-{0} "A1";"A2"
\ar@/^/|-{B} "A2";"A3"
\ar@/^/|-{0} "A3";"A1"
\ar@/^/|-{1} "A1";"A3"
\ar@/^/|-{0} "A3";"A2"
\ar@/^/|-{A\ } "A2";"A1"
\endxy \\
\end{array}
\]
By Lemma \ref{tor-I_i}, it follows that $M_1 \in \mathcal T(I_1)$, but $M_2 \not \in \mathcal T(I_1)$. So $\s{\theta}{\Lambda}{,\mathbf d} \not \subset \mathcal T(I_1)$, hence $\Hom_{\Lambda}(I_1,-)$ does not give an equivalence between $\s{\theta}{\Lambda}{,\mathbf d}$ and $\s{s_1\theta}{\Lambda}{,\mathbf d}$. This means that the assumption that $\theta_1 \neq 0$ in Theorem \ref{simple-reflection} is essential.
\end{ex}

Before proceeding to the general case, we state a result.

\begin{prop}\label{non-change}
Let $\theta \in \Theta$ with $\theta_i  \neq 0$ and $M \in \s{\theta}{\Lambda}{}$. \\
(1) If $\theta_i>0$, $M \simeq \bs^{+}(M)$ if and only if $S_i$ is not a direct summand of $\Soc M$ and $(M,S_i)=0$.\\
(2) If $\theta_i<0$, $M \simeq \bs^{-}(M)$ if and only if $S_i$ is not a direct summand of $M/(MI)$ and $(M,S_i)=0$.
\end{prop}

\begin{proof}
Take any $M \in \s{\theta}{\Lambda}{}$. 
We only show (1) because (2) is similar. 
We assume that $M \simeq \Hom_{\Lambda}(I_i,M)$.
Since $M \in \mathcal T(I_i)$, $\Hom_{\Lambda}(I_i,M) \in \mathcal{Y}(I_i)$, so $M \in \mathcal Y(I_i)$.
Thus $S_i$ is not a direct summand of $\Soc M$ by Lemma \ref{tor-I_i}.
Moreover since $M \simeq \Hom_{\Lambda}(I_i,M)$ and $M \in \mathcal T(I_i)$, by Corollary \ref{change-dim}, we have $\udim M = \udim \Hom_{\Lambda}(I_i,M) = s_i(\udim M)$, which implies $(M,S_i)=0$.

Conversely we assume that $S_i$ is not a direct summand of $\Soc M$ and $(M,S_i)=0$. 
By applying $\Hom_{\Lambda}(-,M)$ to the exact sequence
$0 \to I_i \to \Lambda \to S_i \to 0$,
we have an exact sequence
\[
\Hom_{\Lambda}(S_i,M) \longrightarrow \Hom_{\Lambda}(\Lambda,M) \longrightarrow \Hom_{\Lambda}(I_i,M) \longrightarrow \Ext^1_{\Lambda}(S_i,M).
\]
Thus it is enough to show that $\Hom_{\Lambda}(S_i,M)=0$ and $\Ext^1_{\Lambda}(S_i,M)=0$.
By the assumption, we have $\Hom_{\Lambda}(S_i,M)=0$. 
On the other hand, since $M \in \mathcal{T}(I_i)$, by Lemma \ref{tor-I_i}, $S_i$ is not a direct summand of $M/(MI)$.
Hence by Lemma \ref{CB-form}, we have
\[
\dim \Ext^1_{\Lambda}(S_i,M) = \dim \Hom_{\Lambda}(S_i,M) + \dim \Hom_{\Lambda}(M,S_i) - (S_i,M)= 0.
\]
Therefore the assertion follows.
\end{proof}

\subsection{General case $w = s_{i_{\ell}}\cdots s_{i_1}$}

First we introduce the notation $\bs_i$ for any $i \in Q_0$ as follows:
\begin{itemize}
\item If $\theta_i>0$, then $\bs_i := \bs_i^{+} = \Hom_{\Lambda}(I_i,-) : \s{\theta}{\Lambda}{} \to \s{\theta}{\Lambda}{s_i\theta}$,
\item If $\theta_i<0$, then $\bs_i := \bs_i^{-} = -\otimes_{\Lambda}I_i : \s{\theta}{\Lambda}{} \to \s{\theta}{\Lambda}{s_i\theta}$.
\end{itemize}

Let $w$ be an element of the Coxeter group $W_Q$.
If we take a expression $w = s_{i_{\ell}}\cdots s_{i_1}$, a composition $\bs_{i_{\ell}}\cdots \bs_{i_1}$ of the simple reflection functors gives an equivalence between $\s{\theta}{\Lambda}{}$ and $\s{w\theta}{\Lambda}{}$ for a general $\theta \in \Theta$. 
The next result implies that it does not depend on the choice of the expression of $w$ up to isomorphisms.

\begin{prop}\label{Coxeter-relation}
The functors $\bs_i$ satisfy the Coxeter relation, namely
\begin{enumerate}
\item 
$\bs_i \bs_i \simeq \mathrm{id}$,
\item
$\bs_j \bs_i \simeq \bs_j \bs_i$
\text{ if there is no arrows between $i$ and $j$ in $Q$, }
\item 
$\bs_i \bs_j \bs_i \simeq \bs_j \bs_i \bs_j$ 
\text{ if there is precisely one arrow between $i$ and $j$ in $Q$. }
\end{enumerate}
where $\theta_i \neq 0$, $\theta_j \neq 0$ and $\theta_i+\theta_j \neq 0$.
\end{prop}

\begin{proof}
(1) is trivial. (2) There are 4 cases: (i) $\theta_i>0$, $\theta_j>0$, (ii) $\theta_i>0$, $\theta_j<0$, (iii) $\theta_i<0$, $\theta_j>0$ and (iv) $\theta_i<0$, $\theta_j<0$. For example, if $\theta_i>0$ and $\theta_j<0$, then for a $\theta$-semistable $\Lambda$-module of $M$, by Proposition \ref{rel-tilt} and \ref{red-tilt} we have
\begin{align*}
&\bs_j^- \bs_i^+(M) \simeq \Hom(I_i,M)\otimes I_j \simeq  \Hom(I_i,\Hom(I_j,M\otimes I_j))\otimes I_j \simeq \Hom(I_iI_j,M\otimes I_j) \otimes I_j \\
&\simeq \Hom(I_jI_i,M\otimes I_j))\otimes I_j \simeq \Hom(I_j,\Hom(I_i,M\otimes I_j))\otimes I_j = \Hom(I_i,M\otimes I_j)) = \bs_i^+ \bs_j^-(M).
\end{align*}
The others are similar. (3) There are 6 cases and the proof is similar to (2).
\end{proof}

Thus we denote by $\bw$ the composition $\bs_{i_{\ell}} \cdots \bs_{i_1}$ of simple reflection functors for any expression $s_{i_{\ell}}\cdots s_{i_1}$ of $w$. We call $\bw$ the reflection functor. Summary we have the following result.

\begin{thm}\label{reflection-functor}
For any $w \in W$, if $\theta \in \Theta$ is sufficiently general, then there is a categorical equivalence
\[\xymatrix{
\s{\theta}{\Lambda}{} \ar@<0.5ex>[rr]^{\bw} && \ar@<0.5ex>[ll]^{\bw^{-1}} \s{w\theta}{\Lambda}{}.
}\]
Under this equivalence $S$-equivalence classes are preserved and $\theta$-stable modules correspond to $w\theta$-stable modules.
In particular it induces the equivalence between $\s{\theta}{\Lambda}{,\alpha}$ and $\s{w\theta}{\Lambda}{,w\alpha}$.
\end{thm}

If we impose some assumption, $\bw$ has a explicit description  by using the tilting module $I_w$.

\begin{prop}\label{composition}
For any $w \in W_Q$, take a reduced expression $w=s_{i_{\ell}}\cdots s_{i_1}$. Put $w_j=s_{i_j}\cdots s_{i_1}$ for $j=1,\ldots,\ell$ and $w_0=1$. For any $\theta \in \Theta$, if $(w_{j-1}\theta)_{i_j}>0$  holds for any $j=1,\ldots, \ell$, then we have $\s{\theta}{\Lambda}{} \subset \mathcal T(I_w), \s{w\theta}{\Lambda}{} \subset \mathcal Y(I_w)$ and
\[
\bw = \bs_{i_{\ell}} \cdots \bs_{i_{1}} = \bs^{+}_{i_{\ell}} \cdots \bs^{+}_{i_{1}} = \Hom_{\Lambda}(I_w,-) : \s{\theta}{\Lambda}{} \to \s{w\theta}{\Lambda}{}.
\]
Also if $(w_{j-1}\theta)_{i_j}<0$ holds for any $j=1,\ldots\ell$, then we have  $\s{\theta}{\Lambda}{} \subset \mathcal Y(I_w), \s{w\theta}{\Lambda}{} \subset \mathcal T(I_w)$ and
\[
\bw = \bs_{i_{\ell}} \cdots \bs_{i_{1}} = \bs^{-}_{i_{\ell}} \cdots \bs^{-}_{i_{1}} = -\otimes_{\Lambda}I_w : \s{\theta}{\Lambda}{} \to \s{w\theta}{\Lambda}{}.
\]
\end{prop}

\begin{proof}
Since $s_{i_{\ell}} \cdots s_{i_1}$ is a reduced expression, the assertion follows from Lemma \ref{TY} and Proposition \ref{red-tilt}. 
\end{proof}

\section{Functor induces morphism}
\label{functor}

In this section we prove that an equivalence between categories obtained in the previous section induces a isomorphism of varieties. In view of applications, we prove it in a more general setting. 
Here a variety means a separated reduced scheme of finite type over $K$.
Let $\Var$ be the category of varieties, $\R$ be the category of finitely generated reduced commutative $K$-algebras and $\Set$ the category of sets.
By Yoneda's lemma (cf.~\cite[Proposition VI-2]{EH}), $\Var$ is equivalent to the full subcategory of the category of functors from $\R$ to $\Set$. So a variety $X$ is regarded as a covariant functor $h_X : \R \to \Set$ the so-called functor of points.
For simplicity, for any $K$-algebra $A$, any $A$-module $M$ and any $R \in \R$, $A^R$ and $M^R$ stands for $\Lambda \otimes_K R$ and $M\otimes_K R$ respectively. For any $R \in \R$, $\Max(R)$ denotes the set of maximal ideal of $R$.
King \cite{Ki} provided a notion of families for a $K$-algebra $\Lambda = KQ/\langle R \rangle$ where $(Q,R)$ is a finite quiver with relations. A family of $\theta$-semistable $\Lambda$-module of dimension vector $\alpha$ is defined as a $\theta$-semistable $\Lambda^R$-module of dimension vector $\alpha$ which is finitely generated and locally free over $R$. 
Here a $\Lambda^R$-module $M$ is called $\theta$-semistable if $M \otimes_R k(\m)$ is $\theta$-semistable for any $\m \in \Max(R)$, also is called of dimension vector $\alpha$ if the dimension vector of $M \otimes_R k(\m)$ is $\alpha$ as a $\Lambda$-module for any $\m \in \Max(R)$ where $k(\mathfrak{m}) = R_{\mathfrak m}/{\mathfrak m}R_{\mathfrak m}$.
Similarly $\s{\theta}{\Lambda^R}{,\alpha}$ denotes the category of $\theta$-semistable $\Lambda$-modules of dimension vector $\alpha$ which is finitely generated and locally free over $R$.

\begin{lem}\label{fiber}
Let $R \in \R$. We assume that a finitely generated $R$-module $M$ satisfies the following condition. For any $\m \in \Max(R)$, there exists a non-negative integer $d$ such that
\[
\dim_{k(\mathfrak{m})} M \otimes_R k(\mathfrak{m}) = d.
\]
Then $M$ is locally free over $R$.
Consequently a $\Lambda^R$-module of dimension vector $\alpha$ is locally free over $R$. In particular, such module is flat over $R$.
\end{lem}

\begin{proof}
The function
\[
\psi(\mathfrak p) = \dim_{k(\mathfrak{p})} M \otimes_R k(\mathfrak{p})
\]
is upper-semi-continuous by \cite[Exercise 5.8]{Har}, that is, for any $n \in \mathbb Z$ the set $U_{\geq n} = \{ \mathfrak p \in \Spec R \mid \psi(\mathfrak p) \geq n \}$ is closed. Also trivially the set $U_{\leq n} = \{ \mathfrak p \in \Spec R \mid \psi(\mathfrak p) \leq n \}$ is open.
$\mathrm{Max}(R)$ denotes the set of maximal ideals of $R$. Then by the assumption we have 
\[
\mathrm{Max}(R) \subset U_{\geq d} \text{ and } \mathrm{Max}(R) \subset U_{\leq d}.
\]
Since $R$ is finitely generated over $K$, by \cite[Theorem 5.5]{Ma} it follows that $(0) = \bigcap_{\mathfrak{m} \in \mathrm{Max}(R)}\mathfrak{m}$. So we have $U_{\geq d} = \Spec R = U_{\leq d}$, hence $U_{d} := U_{\geq d} \cap U_{\leq d} = \Spec R$. Therefore by \cite[Exercise 5.8]{Har} $M$ is locally free. 
Moreover, for a Noetherian commutative ring $R$ and a finitely generated $R$-module $M$, being locally free is the same as being flat.
\end{proof}

We denotes by $\F{\theta}{\Lambda}{\alpha}$ the moduli functor with respect to $\theta$-semistable $\Lambda$-modules of dimension vector $\alpha$, namely $\F{\theta}{\Lambda}{\alpha}$ is a covariant functor from $\R$ to $\Set$ defined by 
\[
\F{\theta}{\Lambda}{\alpha}(R) = \left\{\begin{array}{l} \text{ $S$-equivalence classes of $\theta$-semistable $\Lambda^R$-modules of dimension } \\ \text{ vector $\alpha$ which are finitely generated (and locally free) over $R$ }  \end{array}\right\}
\] 
for any $R \in \R$.
By the definition of coarse moduli spaces, for any $\theta \in \Theta$, there is a morphism  
\[
\Ph{\theta}{\Lambda}{\alpha} : \F{\theta}{\Lambda}{\alpha} \longrightarrow \uM{\theta}{\Lambda}{\alpha}
\]
such that $\Ph{\theta}{\Lambda}{\alpha}(K)$ is bijective and, for any variety $X$ and any morphism $\Psi : \F{\theta}{\Lambda}{\alpha} \to h_{X}$, there is a unique morphism $\Omega : \uM{\theta}{\Lambda}{\alpha}$ such that $\Psi = \Omega \circ \Ph{\theta}{\Lambda}{\alpha}$.
In particular if $\M{\theta}{\Lambda}{\alpha}$ is a fine moduli space, then by the definition $\Ph{\theta}{\Lambda}{\alpha}$ is an isomorphism.

Let $(Q,R), (Q',R')$ be finite quivers with relations. Put $\Lambda = KQ/\langle R \rangle$ and $\Gamma = KQ'/R'$. Suppose that there is a functor $F : \s{\theta}{\Lambda}{,\alpha} \to \s{\eta}{\Gamma}{,\beta}$ which preserves S-equivalence classes where $\alpha$ and $\beta$ are dimension vectors. Moreover we assume that the functor $F(R) : \s{\theta}{\Lambda}{,\alpha} \to \Mod \Gamma^R$ is given for each $R \in \R$ such that $F(K) = F$.
Then $F$ defines a map $f(K) : \uM{\theta}{\Lambda}{\alpha}(K) \to \uM{\eta}{\Gamma}{\beta}(K)$. We prove that this map is extended to a morphism of varieties.

\begin{prop}\label{morph}
 If $(-\otimes_RS)\circ F(R) \simeq F(S) \circ (-\otimes_RS)$ holds for each morphism $R \to S$ in $\R$, namely there is a commutative diagram
\[\xymatrix{
\s{\theta}{\Lambda^R}{,\alpha} \ar[r]^{F(R)} \ar[d]_{-\otimes_RS} & \Mod \Gamma^R \ar[d]^{-\otimes_RS} \\
\s{\theta}{\Lambda^S}{,\alpha} \ar[r]^{F(S)} & \Mod \Gamma^S,
}\]
then $F$ is extended to a morphism $f : \uM{\theta}{\Lambda}{\alpha} \to \uM{\eta}{\Gamma}{\beta}$ of functors, therefore a morphism $f : \M{\theta}{\Lambda}{\alpha} \to \M{\eta}{\Gamma}{\beta}$ of varieties.
\end{prop}

\begin{proof}
Let $R \in \R$. Take any $M \in \s{\theta}{\Lambda^R}{,\alpha}$.
For any $\m \in \Max(R)$, we have a natural surjection $R \to k(\m)$.  
So by the assumption we have an isomorphism
\[
F(R)(M) \otimes_R k(\m) \simeq F(k(\m))(M \otimes_R k(\m)).
\]
Since $M \otimes_R k(\m) \in \s{\theta}{\Lambda}{,\alpha}$ and $k(\m) \simeq K$ as a $K$-algebra, it follows that  $F(k(\m))(M \otimes_R k(\m)) = F(M \otimes_R k(\m)) \in \s{\eta}{\Gamma}{,\beta}$. Since $F$ preserves S-equivalence classes, it defines a map $f'(R) : \F{\theta}{\Lambda}{\alpha}(R) \to \F{\eta}{\Gamma}{\beta}(R)$. By the assumption, trivially $f'(R)$ is functorial in $R$, hence we have a morphism $f' : \F{\theta}{\Lambda}{\alpha} \to \F{\eta}{\Gamma}{\beta}$ of functors.
Thus by the definition of coarse moduli spaces, for a composition $\Ph{\eta}{\Gamma}{\beta} \circ f'$, there exists a morphism $f : \uM{\theta}{\Lambda}{\alpha} \to \uM{\eta}{\Gamma}{\beta}$ which makes the following diagram commutative:
\[\xymatrix{
\F{\theta}{\Lambda}{\alpha} \ar[r]^{\Ph{\theta}{\Lambda}{\alpha}} \ar[d]_{f'} & \uM{\theta}{\Lambda}{\alpha} \ar@{..>}[d]^{f} \\
\F{\eta}{\Gamma}{\beta} \ar[r]^{\Ph{\eta}{\Gamma}{\beta}} & \uM{\eta}{\Gamma}{\beta}.
}\]
By Yoneda's lemma, there exists the corresponding morphism $f : \M{\theta}{\Lambda}{\alpha} \to \M{\eta}{\Gamma}{\beta}$ of varieties.
\end{proof}

In particular, we consider the case when $F$ is given as either a hom functor or a tensor functor.
We use the following useful basic facts without a proof.

\begin{lem}\label{basic-iso}
Let $\Lambda$ be a $K$-algebra, $R$ a commutative $K$-algebra and $S$ a commutative $R$-algebra.
\begin{enumerate} 
\def\labelenumi{(\theenumi)} 
\item Let $P,M$ be $\Lambda^R$-modules and $N$ an $R$-module. 
Then there exists a morphism 
\[
\Hom_{\Lambda^R}(P^R,M) \otimes_R N \stackrel{\sim}{\longrightarrow}\Hom_{\Lambda}(P,M \otimes_R N)
\] 
given by $\varphi \otimes n \longmapsto  (x \longmapsto  \varphi(x \otimes 1) \otimes n)$, which are functorial in $P,M$ and $N$.
Moreover if $P$ is finitely generated projective, then it is an isomorphism.
\item For any $\Lambda^{\op}$-module $L$, $\Lambda^R$-module $M$ and $R$-module $N$, there exists an  isomorphism which is functorial in $L,M$ and $N$
\[
M \otimes_{\Lambda^R} (L \otimes_K N) \simeq N \otimes_R M \otimes_{\Lambda} L.
\]
\item Let $P$ be a finitely generated projective $\Lambda^R$-module. Then for any $\Lambda^R$-module $M$ which is flat over $R$, $\Hom_{\Lambda^R}(P,M)$ is also flat over $R$.
\end{enumerate}
\end{lem}

\begin{prop}\label{tensor}
Suppose $F = -\otimes_{\Lambda}L$ for a $\Lambda^{\mathrm{op}}$-module $L$. 
Let $F(R) = -\otimes_{\Lambda^R}L^R$ for each $R \in \R$.
Then $(-\otimes_RS)\circ F(R) \simeq F(S) \circ (-\otimes_RS)$ holds for each morphism $R \to S$ in $\R$.
\end{prop}

\begin{proof}
For any $M \in \s{\theta}{\Lambda^R}{,\alpha}$, by Lemma \ref{basic-iso} (2) we have
\begin{align*}
(M\otimes_{\Lambda^R} L^R) \otimes_RS 
&= (M\otimes_{\Lambda^R} (L\otimes_KR)) \otimes_RS 
\simeq  (M \otimes_{\Lambda} L) \otimes_RS
\simeq S \otimes_S (S \otimes_R M)\otimes_{\Lambda}L \\
&\simeq (S\otimes_RM) \otimes_{\Lambda^S} (L\otimes_KS) 
\simeq (M\otimes_RS) \otimes_{\Lambda^S} L^S.
\end{align*}
\end{proof}

In the hom functor case, we need some assumption.
However it is satisfied in the setting in the previous section and probably in the higher dimensional case, for example, when $\Lambda$ is a $d$-Calabi-Yau algebra and $L$ is a partial tilting module and so on.

\begin{prop}\label{hom}
Suppose $F = \Hom_{\Lambda}(L,-)$ for a $\Lambda$-module $L$ of projective dimension $d$ such that $L$ has a resolution by finitely generated projective modules and $\Ext^i_{\Lambda}(L,M)=0$ for any $M \in \s{\theta}{\Lambda}{,\alpha}$ and $i \geq 1$. 
Let $F(R) = \Hom_{\Lambda^R}(L^R,-)$ for each $R \in \R$.
Then $(-\otimes_RS)\circ F(R) \simeq F(S) \circ (-\otimes_RS)$ holds for each ring homomorphism $R \to S$ in $\R$.
\end{prop}

\begin{proof}
Let $M \in \s{\theta}{\Lambda^R}{,\alpha}$. First we show that $\Ext^i_{\Lambda^R}(L^R,M) = 0$ for any $i \geq 1$.
It is sufficient to show that $\Ext^i_{\Lambda^R}(L^R,M) \otimes_R k(\m) = 0$ for any $\m \in \max(R)$.
By the assumption of $L$, there is a projective resolution 
\begin{eqnarray}\label{resol-long}
0 \stackrel{}{\longrightarrow} P_d \stackrel{f_d}{\longrightarrow} P_{d-1} \stackrel{f_{d-1}}{\longrightarrow} P_{d-2} \stackrel{f_{d-2}}{\longrightarrow} \cdots \stackrel{f_2}{\longrightarrow} P_1 \stackrel{f_1}{\longrightarrow} P_0 \stackrel{f_0}{\longrightarrow} L \to 0
\end{eqnarray}
of $L$ where $P_i$ is finitely generated over $\Lambda$ for any $i = 1,\ldots,d$. For any $i = 0,\ldots,d$, by putting $C_i = \IM f_i$ we have short exact sequences
\begin{align}\label{resol}
0 \stackrel{}{\longrightarrow} C_{i+1} \stackrel{}{\longrightarrow} P_i \stackrel{}{\longrightarrow} C_i \stackrel{}{\longrightarrow} 0.
\end{align}
Note that $C_0 = L$ and $C_d = P_d$.
By applying $\Hom_{\Lambda^R}(-\otimes_K R,M) = (\Hom_{\Lambda^R}(-,M)) \circ (-\otimes_K R)$ to \eqref{resol}, we obtain 
\[
\Ext^j_{\Lambda^R}(C_{i+1}^R,M) \simeq \Ext^{j+1}_{\Lambda^R}(C_i^R,M),
\]
for any $j \geq 1$, so we have $\Ext^1_{\Lambda^R}(C_i^R,M) \simeq \Ext^{i+1}_{\Lambda^R}(L^R,M)$ for any $i = 0,\ldots,d-1$. 
Hence it is enough to show that $\Ext^1_{\Lambda^R}(C_i^R,M) \otimes_R k(\m) = 0$ for any $i= 0,\ldots,d-1$ and any $\m \in \Max(R)$.
For any $\m \in \Max(R)$, by applying $(-\otimes_R k(\m))\circ(\Hom_{\Lambda^R}(-\otimes_K R,M))$ and $\Hom_{\Lambda}(-,M\otimes_R k(\m))$ to the exact sequence \eqref{resol}, by Lemma \ref{basic-iso} (1) we have a commutative diagram
\[\xymatrix{
\Hom_{\Lambda^R}(P_{i}^R,M)\otimes_R k(\m) \ar[r]^{a_{i+1}} \ar[d]^{g_i} & \Hom_{\Lambda^R}(C_{i+1}^R,M)\otimes_R k(\m) \ar[r] \ar[d]^{h_{i+1}} & \Ext^1_{\Lambda^R}(C_{i}^R,M)\otimes_R k(\m) \ar[r] & 0 \\
\Hom_{\Lambda}(P_i,M\otimes_R k(\m)) \ar[r]^{a_{i+1}'} & \Hom_{\Lambda}(C_{i+1},M\otimes_R k(\m)) \ar[r] & \Ext^1_{\Lambda}(C_i,M\otimes_R k(\m)) \ar[r] & 0
}\]
with exact rows where $g_i$ is isomorphic, and
\[
\Ext^{j+1}_{\Lambda}(C_i,M \otimes_R k(\m)) \simeq \Ext^j_{\Lambda}(C_{i+1},M \otimes_R k(\m))
\]
for any $j \geq 1$. 
So, since $C_0 = L$ and $M \otimes_R k(\m) \in \s{\theta}{\Lambda}{,\alpha}$, by the assumption we have $\Ext^1_{\Lambda}(C_i,M\otimes_R k(\m)) \simeq \Ext^{i+1}_{\Lambda}(L,M\otimes_R k(\m)) = 0$ for any $i = 0,\ldots,d-1$.

Now we use an induction on $i$ to prove $\Ext^1_{\Lambda^R}(C_i^R,M) \otimes_R k(\m)=0$. 
We assume that $\Ext^1_{\Lambda^R}(C_j^R,M) \otimes_R k(\m)=0$ holds for any $j = i+1,\ldots,d-1$ and $h_{i+1}$ is isomorphic. Then $\Ext^1_{\Lambda}(C_i,M\otimes_R k(\m)) = 0$ implies $\Ext^1_{\Lambda^R}(C_{i}^R,M) \otimes_R k(\m) = 0$. 
Furthermore, by the induction hypothesis, by applying $\Hom_{\Lambda^R}(-\otimes_K R,M)$ to \eqref{resol-long} we have an exact sequence
\[
0 \to \Hom_{\Lambda^R}(C_j^R,M) \to \Hom_{\Lambda^R}(P_j^R,M) \to \Hom_{\Lambda^R}(P_{j+1}^R,M) \to 0
\] 
for any $j = i,\ldots,d-1$. By applying $-\otimes_R k(\m)$ to it, since $\Hom_{\Lambda^R}(P_j,M)$ are flat over $R$ by Lemma \ref{basic-iso} (3), we  have
$\Tor_{\ell}^R(\Hom_{\Lambda^R}(C_j^R,M),k(\m)) \simeq \Tor_{\ell+1}^R(\Hom_{\Lambda^R}(C_{j+1}^R,M),k(\m))$ for any $j = i\ldots,d-1 $ and $\ell \geq 1$. 
So, since $C_d=P_d$, we have
$\Tor_1^R(\Hom_{\Lambda^R}(C_{i+1}^R,M),k(\m)) \simeq \Tor_{d-i}^R(\Hom_{\Lambda^R}(P_d^R,M),k(\m))=0$.
Thus by Lemma \ref{basic-iso} (1) we have a commutative diagram
\[\xymatrix{
0 \ar[r] & \Hom_{\Lambda^R}(C_{i}^R,M)\otimes_R k(\m) \ar[r]^{b_i} \ar[d]^{h_i} &  \Hom_{\Lambda^R}(P_{i}^R,M)\otimes_R k(\m) \ar[r]^{a_{i+1}} \ar[d]^{g_i} & \Hom_{\Lambda^R}(C_{i+1}^R,M)\otimes_R k(\m) \ar[d]^{h_{i+1}} \\
0 \ar[r] & \Hom_{\Lambda}(C_i,M\otimes_R k(\m)) \ar[r]^{b_i'} & \Hom_{\Lambda}(P_i,M\otimes_R k(\m)) \ar[r]^{a_{i+1}'} & \Hom_{\Lambda}(C_{i+1},M\otimes_R k(\m))
}\]
where each row is exact. Since $g_i,h_{i+1}$ are isomorphic, $h_i$ is also isomorphic. 

In the case $i = d-1$, since $C_d = P_d$, we have $\Ext^1_{\Lambda^R}(C_d^R,M)\otimes_R k(\m) = 0$, and $h_{d}$ is isomorphic. 
Consequently it follows that $\Ext^i_{\Lambda^R}(L^R,M) = 0$ for any $i \geq 1$.

Now we prove that $(-\otimes_RS)\circ F(R) \simeq F(S) \circ (-\otimes_RS)$ holds for each morphism $R \to S$ in $\R$.
Since $\Ext^i_{\Lambda^R}(L^R,M) = 0$ for any $i \geq 1$, 
by applying $\Hom_{\Lambda^R}(-\otimes_K R, M)$ to the exact sequence \eqref{resol-long}, we have an exact sequene
\[
0 \to \Hom_{\Lambda^R}(L^R,M) \to \Hom_{\Lambda^R}(P_0^R,M) \to \cdots \to \Hom_{\Lambda^R}(P_d^R,M) \to 0
\]
Further by applying $-\otimes_R S$ to this sequence, since $\Hom_{\Lambda^R}(P_i^R,M)$ is flat over $R$ for any $i = 0,\cdots,d$ by Lemma \ref{basic-iso}, we have an exact sequence
\[
0 \to  \Hom_{\Lambda^R}(L^R,M)\otimes_RS \to \Hom_{\Lambda^R}(P_0^R,M)\otimes_RS \to \Hom_{\Lambda^R}(P_1^R,M)\otimes_RS.
\]
So by applying $\Hom_{\Lambda^S}(-\otimes_K S,M\otimes_R S)$ to the exact sequence \eqref{resol-long}, we have a diagram
\[\xymatrix{
0 \ar[r]  & \Hom_{\Lambda^R}(L^R,M)\otimes_RS \ar[r] & \Hom_{\Lambda^R}(P_0^R,M)\otimes_RS \ar[r] \ar[d] & \Hom_{\Lambda^R}(P_1^R,M)\otimes_RS \ar[d] \\
0 \ar[r] & \Hom_{\Lambda^S}(L^S,M\otimes_RS) \ar[r] & \Hom_{\Lambda^S}(P_0^S,M\otimes_RS) \ar[r] & \Hom_{\Lambda^S}(P_1^S,M\otimes_RS)
}\]
where each row is exact and, by Lemma \ref{basic-iso} the vertical maps are isomorphism which make the diagram commutative. Therefore we have the required isomorphism $F(R)(M)\otimes_RS \simeq F(S)(M\otimes_RS)$. 
\end{proof}

Combining these results, we have the following result.

\begin{thm}\label{equiv-iso}
Let $L$ be a $(\Gamma,\Lambda)$-bimodule satisfying the condition in Proposition \ref{hom} as a right $\Lambda$-module.
If $\Hom_{\Lambda}(L,-)$ and $-\otimes_{\Lambda}L$ give a categorical equivalence 
\[\xymatrix{
\s{\theta}{\Lambda}{,\alpha} \ar@<0.5ex>[rr]^{\Hom_{\Lambda}(L,-)} && \ar@<0.5ex>[ll]^{-\otimes_{\Gamma}L} \s{\eta}{\Gamma}{,\beta}, 
}\]
then there is an isomorphism $\M{\theta}{\Lambda}{\alpha} \to \M{\eta}{\Gamma}{\beta}$ of varieties.
\end{thm}

\begin{proof}
By Proposition \ref{morph}, \ref{tensor} and \ref{hom}, the functors $\Hom_{\Lambda}(L,-)$ and $-\otimes_{\Lambda}L$ induce morphisms $f : \M{\theta}{\Lambda}{\alpha} \to \M{\eta}{\Gamma}{\beta}$ and $G : \M{\eta}{\Gamma}{\beta} \to  \M{\theta}{\Lambda}{\alpha}$ such that $f\Ph{\theta}{\Lambda}{\alpha} = \Ph{\eta}{\Gamma}{\beta}f'$ and $g\Ph{\eta}{\Gamma}{\beta} = \Ph{\theta}{\Lambda}{\alpha}g'$:
\[\xymatrix{
\F{\theta}{\Lambda}{\alpha} \ar[r]^{\Ph{\theta}{\Lambda}{\alpha}} \ar[d]_{f'} & \uM{\theta}{\Lambda}{\alpha} \ar@{..>}[d]^{f} \\
\F{\eta}{\Gamma}{\beta} \ar[r]^{\Ph{\eta}{\Gamma}{\beta}} & \uM{\eta}{\Gamma}{\beta}.
}
\qquad\xymatrix{
\F{\theta}{\Lambda}{\alpha} \ar[r]^{\Ph{\theta}{\Lambda}{\alpha}} \ar@{<-}[d]_{g'} & \uM{\theta}{\Lambda}{\alpha} \ar@{<..}[d]^{g} \\
\F{\eta}{\Gamma}{\beta} \ar[r]^{\Ph{\eta}{\Gamma}{\beta}} & \uM{\eta}{\Gamma}{\beta}
}\]
where $f'(R),g'(R)$ are maps given by $\Hom_{\Lambda^R}(L,-)$ and $-\otimes_{\Gamma^R}L$.
Since all $\Ph{\theta}{\Lambda}{\alpha}(K)$, $\Ph{\eta}{\Gamma}{\beta}(K)$, $f'(K)$, $g'(K)$ are bijective, we have $g(K)f(K) = id_{h_{\M{\theta}{\Lambda}{\alpha}}(K)}$ and $f(K)g(K) = id_{h_{\M{\eta}{\Gamma}{\beta}}(K)}$.
This implies that $f$ and $g$ give an isomorphism of varieties.
\end{proof}

Now we return to the setting in the previous section.

\begin{cor}\label{pp-equiv-iso}
For any preprojective algebra $\Lambda = K\overline{Q}/\langle R \rangle$, any element $w$ of the Coxeter group $W_Q$ and any sufficiently generic parameter $\theta \in \Theta$, the equivalence $\bw$ induces the isomorphism of algebraic varieties:
\[
\M{\theta}{\Lambda}{\alpha} \stackrel{\sim}{\longrightarrow} \M{w\theta}{\Lambda}{w\alpha}.
\]
\end{cor}

\begin{proof}
The functor $\bw$ is given by a composition of simple reflection functors $\bs_i$. Thus it follows from Theorem \ref{simple-reflection} and \ref{equiv-iso}.
\end{proof}

\section{Kleinian singularity case}
\label{Klein}

In this section, we especially investigate the Kleinian singularity.  
We assume $K$ is an algebraically closed field of characteristic 0.
Let $G$ be a finite subgroup of $SL(2,K)$ of type $\Gamma$, that is, $\Gamma = {A}_m (m \geq 1), {D}_m (m \geq 4), {E}_6, {E}_7$ or ${E}_8$. We denote by $\widetilde{\Gamma}$ the type of the extended Dynkin diagram of $\Gamma$. We denote by $\Lambda$ the preprojective algebra associated to the extended Dynkin quiver of type $\widetilde{\Gamma}$.
Note that the double $\overline{Q}$ of $Q$ is the so-called McKay quiver of $G$ (see Figure \ref{McKay}). 
The vertex set of $\overline{Q}$ is denoted by $\overline{Q}_0 = \{0,1,\ldots,n \}$ where $0$ corresponds to the trivial representation of $G$.

\subsection{Moduli space and Parameter space}

First we recall the relation between moduli spaces of $\Lambda$-module and the Kleinian singularity $\mathbb A^2/G$, and the chamber structure of the parameter space $\Theta$ and the Weyl group $W$.

Let $\mathbf d$ be the dimension vector whose entries are the fimensions of irreducible representations of $G$ (see Figure \ref{McKay}).
Since we are especially interested in moduli spaces of $\Lambda$-modules of dimension vector $\mathbf d$, we define a subset $\Theta_{\mathbf d}$ of $\Theta$ as follows:
\[
\Theta_{\mathbf d} = \{ \theta \in \Theta \mid \theta(\mathbf d)=0 \}.
\]
Then $\theta \in \Theta_{\mathbf d}$ is called {\it generic} if any $\theta$-semistable module is $\theta$-stable.
It is known that, for any $\theta \in \Theta_{\mathbf d}$, the moduli space $\M{\theta}{\Lambda}{\mathbf d}$ gives a partial resolution of the Kleinian singularity $\mathbb A^2/G$. Moreover if $\theta$ is generic, the next result is well-known. 

\begin{thm}[\cite{Kr},\cite{CS},\cite{BKR}]
If $\theta \in \Theta_{\mathbf d}$ is generic, then $\M{\theta}{\Lambda}{\mathbf d}$ is isomorphic to the minimal resolution of the Kleinian singularity $\mathbb A^2/G$ via a natural projective morphism
\[
\M{\theta}{\Lambda}{\mathbf d} \to \M{\mathbf 0}{\Lambda}{\mathbf d} \cong \mathbb A^2/G.
\]
\end{thm}

Historically, for the first time Kronheimer proved the above theorem when $K = \mathbb C$ by constructing it as a hyper-K\"ahler quotient, and next Cassens-Slodowy interpreted it in terms of GIT quotients. On the other hand Bridgeland-King-Reid proved it in a more general setting and their method is valid for any algebraically closed field of characteristic 0.

Next we recall the chamber structure of $\Theta_{\mathbf d}$ given in \cite{Kr,CS}. We prepare some notations about root systems (cf.~\cite{H}).
We write $\widetilde X_{*} = \mathbb Z^{Q_0}$ which is regarded as the affine root lattice of type $\widetilde{\Gamma}$ with a symmetric bilinear form $(-,-)$ defined in Section \ref{pp}. In this case we have 
\[
(\e_i,\e_j) = 
\begin{cases}
2 & i=j \\
-1 & \text{$i$ and $j$ are adjacent vertices in $Q$} \\
0 & \text{$i$ and $j$ are adjacent vertices in $Q$}
\end{cases}.
\]
Since $\mathbf d \in \widetilde X_{*}$ is a minimal imaginary root of $\widetilde X_*$, namely $(\alpha,\mathbf d)=0$ holds for any $\alpha \in \widetilde X_{*}$ and any $\mathbf d' \neq 0$ with such property is written as $\mathbf d'=m\mathbf d$ for some $m \in \mathbb Z$, the quotient lattice 
\[
X_{*} := \widetilde X_{*}/\mathbb Z \mathbf d
\]
becomes the finite root lattice of type $\Gamma$ with the induced bilinear form, again we denote it by $(-,-)$.
We denote the image of $\alpha \in \widetilde X_{*}$ by $\overline{\alpha}$. 
Then $\ove_1,\ldots,\ove_n$ form a basis of $X_{*}$ since $\ove_0=-d_1\ove_1-\cdots-d_n\ove_n$ holds. The dual of $X_{*}$ is given as the sublattice
\[
X^{*} := \{ \theta \in \widetilde X^{*} \mid \theta(\mathbf d) = 0 \}
\] 
of $\widetilde X^*$.
For any $\theta \in X^{*}$, since $\theta(\mathbf d)=0$, we can define $\theta(\overline{\alpha}) := \theta(\alpha)$ for any $\overline{\alpha} \in X_{*}$.
For the finite root lattice $X_*$, let $\Phi$ be the finite root system, $\Delta = \{ \ove_1,\ldots,\ove_n \} \subset \Phi$ a simple root system of $\Phi$ and $\Pi$ (resp.~$-\Pi$) the positive (resp.~negative) root system corresponding to $\Delta$ i.e. $\Pi = \Phi \cap \mathbb Z_{\geq 0}\Delta$. 
Let $W$ be the finite Weyl group associated to the finite root system $\Phi$, which is a finite group generated by simple reflections $s_1,\ldots,s_n$ where $s_i$ is defined by $s_i(\alpha) = \alpha -(\alpha,\ove_i)\ove_i$ for $\alpha \in \Phi$. 
Note that $W$ is regarded as a subgroup of the Coxeter group $W_Q$.
For any element $w \in W$, $w\Delta := \{ w\ove_1,\ldots,w\ove_n \}$ is also a simple root system of $\Phi$.

Now we describe a chamber structure of $\Theta$.
Let $\Theta^{\mathrm{gen}}$ be a subset of $\Theta$ consisting of generic parameters. 
Each connected component in $\Theta^{\mathrm{gen}}$ is called a GIT chamber. $\theta$ and $\theta'$ are contained in the same GIT chamber if and only if $\s{\theta}{\Lambda}{} = \s{\theta'}{\Lambda}{}$ holds.
On the other hand, for any element $w \in W$ the subset
\[
C(w) = \{ \theta \in \Theta \mid \theta(\alpha)>0 \text{ for any } \alpha \in w\Delta \},
\]
of $\Theta$ is called a Weyl chamber.

\begin{prop}[\cite{Kr,CS}]
For any $\theta \in \Theta$, $\theta$ is generic if and only if $\theta(\alpha) \neq 0$ for any   real root $\alpha \in \widetilde X^{*}$ strictly between $0$ and $\mathbf d$, equivalently if and only if $\theta(\alpha)\neq 0$ for any $\alpha \in \Phi$.
\end{prop}

Thus GIT chambers coincide with Weyl chambers: 
\[
\Theta^{\mathrm{gen}} = \coprod_{w \in W} C(w).
\]
We call them just chambers. It is known that $W$ acts on the set of chambers simply transitive (\cite{H}).

In the rest of this section, we only consider the generic parameter.
Then the category $\s{\theta}{\Lambda}{,\mathbf d}$ and the moduli space $\M{\theta}{\Lambda}{\mathbf d}$ are classified by the element of $W$, thus we give the following definition.

\begin{df}
For any $\theta \in C(w)$, we write $\mathcal S_w = \s{\theta}{\Lambda}{,\mathbf d}$ and $\mathcal M_w = \M{\theta}{\Lambda}{\mathbf d}$.
\end{df}

The next fact is used in the following.

\begin{lem}\label{ref-length}
Let $w \in W$ and $i \in \{1,\ldots,n\}$. Then the following are equivalent;
\begin{enumerate}
\item $\ell(s_iw) > \ell(w)$,
\item $\ove_i \in w\Pi$, equivalently $-\ove_i \in s_iw\Pi$,
\item $\theta_i > 0$  for any $\theta \in C(w)$,
\item $\theta_i < 0$  for any $\theta \in C(s_iw)$,
\end{enumerate}
\end{lem}

\begin{proof}
(1) $\Leftrightarrow$ (2) follows from \cite[1.6]{H}. 
By the definition of chambers the rest is clear.
\end{proof}

\subsection{Description of the Reflection functor}

Next we revisit the reflection functor in Kleinian singularity case.
First we consider the simple reflection case.

\begin{prop}\label{Klein-simple-reflection}
For any $w \in W$, if $\ell(w)<\ell(s_iw)$, then we have $\mathcal S_w \subset \mathcal T(I_i)$ and $\mathcal S_{s_iw} \subset \mathcal Y(I_i)$, and  there is a categorical  equivalence 
\[
\xymatrix{
\mathcal S_{w} \ar@<0.5ex>[rrr]^{\bs_i^+ = \Hom_{\Lambda}(I_i,-)} & & & 
\mathcal S_{s_iw} \ar@<0.5ex>[lll]^{\bs_i^- = -\otimes_{\Lambda}I_i}.
}
\]
\end{prop}

\begin{proof}
If we take a $\theta \in C(w)$, then by Lemma \ref{ref-length} we have $\theta_i>0$. So the assertion follows from Theorem \ref{simple-reflection}.
\end{proof}

By virtue of Proposition \ref{Coxeter-relation}, for any $w \in W$, the corresponding reflection functor $\bw = \bs_{i_{\ell}}\cdots\bs_{i_1}$ does not depend on an expression $s_{i_{\ell}}\cdots s_{i_1}$ of $w$ up to isomorphisms, so we especially choose a route through $\overline{C(1)}$. First we observe the  relation between $\mathcal S_1$ and $\mathcal S_w$ for any $w \in W$.

\begin{prop}\label{Klein-basic-reflection}
For any $w \in W$, we have $\mathcal S_1 \subset \mathcal T(I_w)$ and $\mathcal S_w \subset \mathcal Y(I_w)$ and there is a categorical equivalence
\[
\xymatrix{
\mathcal S_{1} \ar@<0.5ex>[rrr]^{\bw \simeq \Hom_{\Lambda}(I_w,-)} & & & 
\mathcal S_{w} \ar@<0.5ex>[lll]^{\bw^{-1} \simeq -\otimes_{\Lambda}I_w}.
}
\]
\end{prop}

\begin{proof}
Let $\theta \in C(1)$.
Take any reduced expression $s_{i_{\ell}}\cdots s_{i_1}$ of $w$. If we write $w_j = s_{i_j}\cdots s_{i_1}$ for any $j = 1,\ldots, \ell$ and $w_0=1$, then  $\ell(s_{j+1}w_j) > \ell(w_j)$ holds for any $j=0,\ldots,\ell-1$. So by Lemma \ref{ref-length}, we have $(w_{j}\theta){i_{j+1}} > 0$ for any $j=0,\ldots,\ell-1$. Hence the assertion follows from Proposition \ref{composition}.
\end{proof}

\begin{thm}\label{Klein-reflection}
For any $w_1,w_2 \in W$, there is a categorical equivalence
\[
\xymatrix{
\mathcal S_{w_1} \ar@<0.5ex>[rrrr]^{\bw_2\bw_1^{-1} \simeq \Hom_{\Lambda}(I_{w_2},-\otimes_{\Lambda}I_{w_1})} &&&& 
\mathcal S_{w_2} \ar@<0.5ex>[llll]^{\bw_1\bw_2^{-1} \simeq \Hom_{\Lambda}(I_{w_1},-\otimes_{\Lambda}I_{w_2})},
}
\]
\end{thm}

\begin{proof}
By Proposition \ref{Klein-basic-reflection}, we have categorical equivalenecs
\[
\xymatrix{
\mathcal S_{w_1} \ar@<0.5ex>[rr]^{- \otimes_{\Lambda} I_{w_1}}& & 
\mathcal S_{1}  \ar@<0.5ex>[ll]^{\Hom_{\Lambda}(I_{w_1},-)} \ar@<0.5ex>[rr]^{\Hom_{\Lambda}(I_{w_2},-)} & & 
\mathcal S_{w_2} \ar@<0.5ex>[ll]^{- \otimes_{\Lambda} I_{w_2}}.
}
\]
Therefore the assertion follows.
\end{proof}

\begin{cor}
For any $w_1,w_2 \in W$, $\mathcal M_{w_1}$ is isomorphic to $\mathcal M_{w_2}$ as a variety, which is given by $\bw_2\bw_1^{-1}$ and $\bw_1\bw_2^{-1}$. 
\end{cor}

\begin{proof}
It follows from corollary \ref{pp-equiv-iso}.
\end{proof}

\begin{rmk}
We only considered the case when $\theta \in \Theta_{\mathbf d}$ is generic. However we can prove similar results for any $\theta \in \Theta$ and any dimension vector $\alpha$ if we impose some assumptions on $\theta$ to use Theorem \ref{simple-reflection}.
\end{rmk}

\subsection{Properties of $S_i^w = \RHom_{\Lambda}(I_w,S_i)$}
\label{S_i^w}

In the rest of this section, we study properties of modules $M$ contained in $\mathcal S_w$ by using the obtained results.
We consider the following two problems. 

\begin{enumerate}
\item Is there an equivalent condition that being $M \in \mathcal S_w$?
\item Is there a characterization of the exceptional curves on $\mathcal M_w$?
\end{enumerate}

The complexes defined below will play an important role to solve the above problems. 
Recall that $S_i$ denotes the simple $\Lambda$-module corresponding to a vertex $i \in Q_0$ and $\e_i= \udim S_i$. 
For any $w \in W$ and $i \in Q_0$, $S_i^w$ denotes a complex $\RHom_{\Lambda}(I_w,S_i)$.

\begin{lem}\label{root-dim}
Precisely either $S_i \in \mathcal T(I_w)$ or $S_i \in \mathcal F(I_w)$ holds. Moreover
\begin{enumerate}
\item $S_i \in \mathcal T(I_w)$ if and only if $w \e_i \in \Pi$. In this case $\udim \Hom_{\Lambda}(I_w,S_i) = w\e_i$.
\item $S_i \in \mathcal F(I_w)$ if and only if $w \e_i \in -\Pi$. In this case $\udim \Ext^1_{\Lambda}(I_w,S_i) = -w\e_i$.
\end{enumerate}
\end{lem}

\begin{proof}
By Lemma \ref{IRprivate}, precisely either $I_w \otimes_{\Lambda} S_i=0$ or $\Tor_1^{\Lambda}(I_w,S_i)=0$ holds. However, by Lemma \ref{isom1}, we have $\Hom_{\Lambda}(I_w,S_i) \simeq D(I_w \otimes_{\Lambda} S_i)$ and $\Ext^1_{\Lambda}(I_w,S_i) \simeq D(\Tor_1^{\Lambda}(I_w,S_i)$, so the first assertion follows. 
Thus the rest follows from the equality
\[
w\e_i = [S_i^w] = \udim (\Hom_{\Lambda}(I_w,S_i)) - \udim (\Ext^1_{\Lambda}(I_w,S_i))
\] 
which is given by Theorem \ref{tilt-Gro}
\end{proof}

By Lemma \ref{root-dim} we have
\[
S_i^w = 
\begin{cases}
\Hom_{\Lambda}(I_w,S_i) & \text{ if $w\e_i \in \Pi$, } \\
\Ext^1_{\Lambda}(I_w,S_i)[-1] & \text{ if $w\e_i \in -\Pi$} \\
\end{cases}
\]
and $[S_i^w] = w\e_i$.
So in the case $w\e_i \in -\Pi$, $S_i^w[1]$ stands for $\Ext^1_{\Lambda}(I_w,S_i)$. 
For any $w \in W$, the dimension vectors $[S_1^w],\ldots,[S_n^w]$ induce the simple root system $w\Delta = \{ w\ove_1,\ldots, w\ove_n \}$.
Also we note that any $S_i^w$ is a 2-spherical object in the derived category of $\Mod(\Lambda)$ and the collection $\{ S_1^w, \ldots, S_n^w \}$ is a $\Gamma$-configuration (cf.~\cite{BT}), though we do not use these properties here.

\subsection{Homological interpretation of the stability condition}

We give an equivalent condition that being $M \in \mathcal S_w$ for any $w \in W$.
Fix a vertex $v \in Q_0$. Then $\Lambda$-module $M$ is said to be $v$-generated if the dimension of $M e_v$ is 1 and $M$ is generated by an element of $M e_v$. For a nilpotent $\Lambda$-module $M$, $M$ is $v$-generated if and only if $M/MI \simeq S_v$. For any $M \in \mathcal S_w$, the dimension of $M e_0$ is $d_0 = 1$.
For the identity element $1 \in W$, the following holds.

\begin{lem}\label{0-gen}
For any $\Lambda$-module $M$ of dimension vector $\mathbf d$, the following are equivalent.
\begin{enumerate}
\item $M \in \mathcal S_1$.
\item $M$ is 0-generated.
\item $\Hom_{\Lambda}(M,S_i) = 0$ holds for all $i=1,\ldots,n$.
\end{enumerate}
\end{lem}

\begin{proof}
(1) $\Leftrightarrow$ (2) is well known result (see \cite[Exercise 4.12]{C}). 
We show (2) $\Leftrightarrow$ (3). 
By Proposition \ref{cat_fd}, $M$ is either simple or nilpotent. If $M$ is simple, then $M$ satisfies both of (2) and (3). If $M$ is nilpotent, then the assertion follows from an isomorphism $\Hom_{\Lambda}(M,S_i) \cong \Hom_{\Lambda}(M/MI,S_i)$.
\end{proof}

Next we consider the general case.  
The next is the main result in this subsection.

\begin{thm}\label{thm:ch-s}
For any $w \in W$ and any $\Lambda$-module $M$ of dimension vector $\mathbf d$, the following are equivalent.
\begin{enumerate}
\item $M \in \mathcal S_w$.
\item For all $i=1,\ldots,n$, the following hold.
\[
\begin{cases}
\Hom_{\Lambda}(M,S_i^w)=0 & \text{ if $w\e_i \in \Pi$},\\
\Hom_{\Lambda}(S_i^w[1],M)=0 & \text{ if $w\e_i \in -\Pi$}.
\end{cases}
\]
\end{enumerate}
\end{thm}

To prove Theorem \ref{thm:ch-s}, we need the next technical lemma.

\begin{lem}\label{lem:ch-s}
For a $w \in W$, we take a a reduced expression $w=s_{i_{\ell}} \cdots s_{i_1}$. Then there exists $i \in \{1,\ldots,n \}$ such that $w\e_i \in -\Pi$ and $\Hom_{\Lambda}(S_i^w[1],S_{i_{\ell}}) \neq 0$. In particular $S_i^w[1]/S_i^w[1]I$ contains $S_{i_{\ell}}$ as a direct summand.
\end{lem}

\begin{proof}
By Lemma \ref{tor-I_i}, $S_{i_{\ell}} \in \mathcal Y(I_{i_{\ell}})$ holds, so we have $S_{i_{\ell}} \otimes_{\Lambda} I_{i_{\ell}} = 0$. 
Hence we have $S_{i_{\ell}}\otimes_{\Lambda}I_w = S_{i_{\ell}}\otimes_{\Lambda} I_{i_{\ell}}\otimes_{\Lambda} I_{s_{i_{\ell}}w} = 0$. 
Thus we have $S_{i_{\ell}} \in \mathcal{X}(I_w)$ and $\Tor^{\Lambda}_1(S_{i_{\ell}},I_w) \in \mathcal{F}(I_w)$. 
By Corollary \ref{change-dim}, we have $\udim \Tor_1^{\Lambda}(S_{i_{\ell}},I_w) = -[S_{i_{\ell}}\Ltensor_{\Lambda}I_w] = -w\e_{i_{\ell}}$, hence $S_0$ does not appear in composition factors of $\Tor^{\Lambda}_1(S_{i_{\ell}},I_w)$. 
Since $\Tor^{\Lambda}_1(S_{i_{\ell}},I_w)$ is nilpotent, there exists $i \in \{1, \ldots, n\}$ such that $\Hom_{\Lambda}(S_i,\Tor^{\Lambda}_1(S_{i_{\ell}},I_w)) \neq 0$. This implies $S_i$ is a submodule of $\Tor^{\Lambda}_1(S_{i_{\ell}},I_w)$.
Since $\mathcal{F}(I_w)$ is closed under submodules, $S_i \in \mathcal{F}(I_w)$. By Lemma \ref{root-dim}, $w\e_i \in -\Pi$.
Moreover we have
\begin{align*}
\Hom_{\Lambda}(S_i^w[1],S_{i_{\ell}}) 
&\simeq \Hom_{\mathcal D}(\RHom_{\Lambda}(I_w,S_i[1]),S_{i_{\ell}})
\simeq \Hom_{\mathcal D}(S_i[1],\Tor^{\Lambda}_1(I_w,S_{i_{\ell}})[1]) \\
&\simeq \Hom_{\Lambda}(S_i,\Tor^{\Lambda}_1(I_w,S_{i_{\ell}})) \neq 0.
\end{align*}
\end{proof}

\begin{proof}[Proof of Theorem \ref{thm:ch-s}]
Take any $\Lambda$-module $M$ of dimension vector $\mathbf d$. 
If $M$ is simple, then the assertion is trivial. 
So we suppose that $M$ is nilpotent.
Put $\ell = \ell(w)$.
We prove the assertion by induction on $\ell(w)$. 
In the case $\ell(w)=0$, since $w=1$, the assertion follows from Lemma \ref{0-gen}. 
We assume that $\ell > 0$ and the assertion holds for any $w'$ with $\ell(w')<\ell$.
Take a reduced expression $w=s_{i_{\ell}} \cdots s_{i_1}$. 
We put $w'=s_{i_{\ell}}w=s_{i_{\ell-1}} \cdots s_{i_1}$.
We remark that $\ell(w)>\ell(s_iw)$ holds.

Now we show (1) $\Rightarrow$ (2). Assume that $M \in \mathcal S_w$. 
Then by Proposition \ref{Klein-simple-reflection}, we have $M \in \mathcal Y(I_{i_{\ell}})$ and $N := M \otimes_{\Lambda}I_{i_{\ell}} \in \mathcal S_{w'}$. So by the induction hypothesis, for all $i=1,\ldots,n$, $N$ satisfies
\[
\begin{cases}
\Hom_{\Lambda}(N,S_i^{w'})=0 & \text{ if $w'\e_i \in \Pi$},\\
\Hom_{\Lambda}(S_i^{w'}[1],N)=0 & \text{ if $w'\e_i \in -\Pi$}.
\end{cases}
\]

\noindent(i) The case $w' \e_i \in \Pi$ and $\e_i \neq \e_{i_{\ell}}$. In this case $w\e_i = s_iw'\e_i \in \Pi$ holds, so by Lemma \ref{root-dim}, $S_i \in \mathcal{T}(I_{w'})$.
Thus by Proposition \ref{red-tilt}, we have $\Hom_{\Lambda}(M,S_i^w) \simeq \Hom_{\mathcal D}(M,\RHom_{\Lambda}(I_{i_{\ell}},S_i^{w'})) \simeq \Hom_{\mathcal D}(M\Ltensor_{\Lambda}I_{i_{\ell}},S_i^{w'}) \simeq \Hom_{\Lambda}(N,S_i^{w'})=0$.

\noindent(ii) The case $w' \e_i \in -\Pi$. In this case $w\e_i \in s_{i_{\ell}}w'\e_i \in -\Pi$ holds, so by Lemma \ref{root-dim} we have $S_i \in \mathcal F(I_w)$.
Thus by Proposition \ref{red-tilt}, we have $\Hom_{\Lambda}(S_i^w,M) \simeq \Hom_{\mathcal D}(\RHom_{\Lambda}(I_{i_{\ell}},S_i^{w'}),M) \simeq \Hom_{\mathcal D}(S_i^{w'},M\Ltensor_{\Lambda}I_{i_{\ell}}) \simeq \Hom_{\Lambda}(S_i^{w'}[1],N)=0$.

\noindent(iii) In the case $w' \e_i=\e_{i_{\ell}}$. In this case $w\e_i = s_{i_{\ell}}w'\e_i = -\e_{i_{\ell}}$, so by Lemma \ref{root-dim} we have $S_i^w[1] \simeq S_{i_{\ell}}$.
By Lemma \ref{tor-I_i}, $M \in \mathcal Y(I_i)$ implies that $\Hom_{\Lambda}(S_i,M)=0$. Thus we have $\Hom_{\Lambda}(S_i^w[1],M)=0$.

Since (i),(ii) and (iii) cover all cases, $M$ satisfies the condition (2). 

Next we show (2) $\Rightarrow$ (1).
Assume that $M$ satisfies the condition (2). 
By Lemma \ref{lem:ch-s}, there exists $i \in \{1,\ldots,n\}$ such that 
$S_i^w[1]/S_i^w[1]I$ contains $S_{i_{\ell}}$ as a direct summand.
The assumption $\Hom_{\Lambda}(S_i^w[1],M)=0$ implies $\Hom_{\Lambda}(S_{i_{\ell}},M)=0$, hence $S_i$ is not a direct summand of $\Soc(M)$. Thus by Lemma \ref{tor-I_i} we have $M \in \mathcal{Y}(I_{i_{\ell}})$.
By repeating a similar argument in the above, we see that $N := M \otimes_{\Lambda} I_{i_{\ell}}$ satisfies the condition (2). 
So by the induction hypothesis, we have $N\in \mathcal S_{w'}$. 
Therefore by Proposition \ref{Klein-simple-reflection}, $M =  \in \mathcal S_w$.
\end{proof}

\subsection{Analogues of the McKay correspondence}

We give a characterization of the exceptional curves on $\mathcal M_w$ for any $w \in W$.

Crawley-Boevey \cite{CB} observed that $\mathcal M_1$ is identified with $G$-Hilbert scheme via the Morita equivalence between $\Lambda$ and the skew group ring.
Ito and Nakamura \cite{IN} explained the McKay correspondence, which is a one-to-one correspondence between the set of exceptional curves on the minimal resolution of $\mathbb A^2/G$ and the set of non-trivial irreducible representation of $G$, by using the $G$-Hilbert scheme. 
Crawley-Boevey \cite{CB} reformulated it in terms of $\Lambda$-modules as follows. 

\begin{thm}[{\cite[Theorem 2]{CB}}]\label{thm:CB}
Let $N \in \mathcal S_1$. Then the socle of $N$ has at most two simple summands, and if two, they are not isomorphic. If $i \in \{ 1,\ldots,n \}$, then
\[
E_i := \{ N \in \mathcal S_1 \mid S_i \text{ is a submodule of } N \}/\simeq
\] 
is a closed subset of $\mathcal M_1$ isomorphic to $\mathbb P^1_K$. Moreover $E_i$ meets $E_j$ if and only if $i$ and $j$ are adjacent in $Q$, and in this case they meet at only one point.
\end{thm}

We generalize Theorem \ref{thm:CB} for any $w \in W$. For all $i \in \{ 1,\ldots,n \}$, we define a subset $E_i^w$ of $\mathcal M_w$ by
\[
E_i^w := \{ \Hom_{\Lambda}(I_w,M) \mid M \in E_i \}/\simeq.
\]

\begin{prop}\label{E_i^w}
If $i \in \{ 1,\ldots,n \}$, then $E_i^w$ is a closed subset of $\mathcal M_w$ isomorphic to $\mathbb P^1_K$. Moreover $E_i^w$ meets $E_j^w$ if and only if $i$ and $j$ are adjacent in $Q$, and in this case they meet at only one point.
\end{prop}

\begin{proof}
It follows immediate from Theorem \ref{equiv-iso} and \ref{thm:CB}.
\end{proof}

Now we state a main result in this subsection.
Although each exceptional curve $E_i$ is characterized by a simple module $S_i$, each exceptional curve $E_i^w$ is characterized by $S_i^w$.

\begin{thm}\label{hom-des-main}
We take any $w \in W$ and $i \in \{ 1,\ldots, n \}$.
For any $M \in \mathcal S_w$, the following hold.
\begin{enumerate}
\item If $w\e_i \in \Pi$, then $M \in E^w_i$ if and only if $S_i^w$ is a submodule of $M$.
\item If $w\e_i \in -\Pi$, then $M \in E^w_i$ if and only if $S_i^w[1]$ is a factor module of $M$.
\end{enumerate}
\end{thm}

In the rest we prove Theorem \ref{hom-des-main}.

\begin{lem}\label{L}
For any $N \in E_i$, there exist non-split exact sequences
\begin{eqnarray}
0 \longrightarrow S_i \longrightarrow L_i^{+} \longrightarrow N \longrightarrow 0, \label{exact3} \\ 
0 \longrightarrow S_i \longrightarrow N \longrightarrow L_i^{-} \longrightarrow 0 \label{exact4}
\end{eqnarray}
such that $L_i^+$ and $L_i^-$ are 0-genrated nilpotent $\Lambda$-module of dimension vector $\mathbf d + \e_i$ and $\mathbf d - \e_i$ respectively.
\end{lem}

\begin{proof}
By Theorem \ref{thm:CB} we have $\dim \Hom_{\Lambda}(S_i,N) = 1$. So by Lemma \ref{CB-form} we have $\dim \Ext^1_{\Lambda}(N,S_i)=1$. We denote by $L_i^+$  the module corresponding to a non-zero element in $\Ext^1_{\Lambda}(N,S_i)$. Also we denote by $L_i^-$ the cokernel of a inclusion $S_i \to N$.
Then it is trivial that the exact sequences \eqref{exact3} and \eqref{exact4} exist and $\udim L_i^+ = \mathbf d + \e_i$ and $\udim L_i^- = \mathbf d - \e_i$. Moreover since these are real roots, \cite[Lemma 2]{CB} claims $L_i^+$ and $L_i^-$ are 0-generated and nilpotent.
\end{proof}

\begin{lem}\label{SL}
For any 0-generated $\Lambda$-module $M$, it is contained in $\mathcal T(I_w)$ for any $w \in W$. 
In particular, $L_i^{+},L_i^{-} \in \mathcal{T}(I_w)$ for any $w \in W$ and $i \in \{1,\ldots,n \}$. 
\end{lem}

\begin{proof}
For any a 0-generated $\Lambda$-module $M$,
$\Ext_{\Lambda}^1(I_w, M) \simeq \Ext_{\Lambda}^2(\Lambda/I_w,M) \simeq D\Hom_{\Lambda}(M, \Lambda/I_w)$ holds. 
Since $S_0$ doesn't appear in composition factors of $\Lambda/I_w$, $\Hom_{\Lambda}(M, \Lambda/I_w)=0$. Thus the assertion follows. 
\end{proof}

For any $i = 1,\ldots, n$, we put $(L_i^+)^w := \Hom_{\Lambda}(I_w,L_i^+)$ and $(L_i^-)^w := \Hom_{\Lambda}(I_w,L_i^-)$.

\begin{lem}\label{L^w}
For any $M \in E_i^w$, there exist non-split exact sequences:
\begin{enumerate}
\item If $w\e_i \in \Pi$, 
\[
0 \longrightarrow S_i^w \longrightarrow M \longrightarrow (L_i^{-})^w \longrightarrow 0.
\]
\item If $w\e_i \in -\Pi$, 
\[
0 \longrightarrow (L_i^{+})^w \longrightarrow M \longrightarrow S_i^w[1] \longrightarrow 0.
\]
\end{enumerate}
\end{lem}

\begin{proof}
By the definition, there is $N \in E_i$ with $M = \Hom_{\Lambda}(I_w,N)$. 
Thus by applying $\Hom_{\Lambda}(I_w,-)$ to the exact sequence \eqref{exact3} and \eqref{exact4} in Lemma \ref{L}, we have exact sequences
\begin{eqnarray*}
0 \to \Hom_{\Lambda}(I_w,S_i) \to \Hom_{\Lambda}(I_w,L_i^+) \to \Hom_{\Lambda}(I_w,N) \to \Ext^1_{\Lambda}(I_w,S_i) \to \Ext^1_{\Lambda}(I_w,L_i^+), \\
0 \to \Hom_{\Lambda}(I_w,S_i) \to \Hom_{\Lambda}(I_w,N) \to \Hom_{\Lambda}(I_w,L_i^-) \to \Ext^1_{\Lambda}(I_w,S_i) \to \Ext^1_{\Lambda}(I_w,L_i^-).
\end{eqnarray*}
By Lemma \ref{SL} $L_i^+,L_i^- \in \mathcal T(I_w)$, so the assertion follows. 
\end{proof}

\begin{proof}[Proof of Theorem \ref{hom-des-main}]
Take any $M \in \mathcal S_w$. Then by the equivalence in Proposition \ref{Klein-basic-reflection}, there is $N \in \mathcal S_1$ such that $M = \Hom_{\Lambda}(I_w,N)$. First we consider in the case $w\e_i \in \Pi$. If $M \in E_i^w$, then by Lemma \ref{L^w} $S_i^w$ is a submodule of $M$. Conversely if $S_i^w$ is a submodule of $M$, by Proposition \ref{Klein-basic-reflection} and Lemma \ref{root-dim}, we have $N,S_i \in \mathcal T(I_w)$, thus $\Hom_{\Lambda}(S_i,N) \simeq \Hom_{\Lambda}(S_i^w,M) \neq 0$. So $S_i$ is a submodule of $N$, hence $N \in E_i$, therefore $M \in E_i^w$.
Next we consider the case $w\e_i \in -\Pi$. If $M \in E_i^w$, then by Lemma \ref{L^w} $S_i^w[1]$ is a factor module of $M$. Conversely if $S_i^w[1]$ is a factor module of $M$, there is a non-split exact sequence $0 \to X \to M \to S_i^w[1] \to 0$. By applying $-\otimes_{\Lambda}I_w$ to it, since $M,X \in \mathcal Y(I_w)$ and $S_i^w[1] \in \mathcal X(I_w)$, we have an non-split exact sequence $0 \to S_i \to X \otimes_{\Lambda}I_w \to N \to 0$. By Lemma \ref{CB-form} we have $\dim \Hom_{\Lambda}(S_i,N) = \dim \Ext^1_{\Lambda}(N,S_i) \neq 0$. Hence $N \in E_i$, therefore $M \in E_i^w$.
\end{proof}

\section{Example}
\label{example}

Let $G$ be a finite subgroup of $SL(2,K)$ of order three which is generated by $\sigma = \diag(\epsilon,\epsilon^2)$ where $\epsilon$ is a primitive third root of unity. Then the McKay quiver $Q$ of $G$, a preprojective relation $R$  and the dimension vector $\mathbf d$ of the irreducible representations are given by
\[
Q=
\begin{array}{c}
\def\alphanum{\ifcase\xypolynode \or 0 \or 1 \or 2 \fi}
\xy 
\xygraph{
!{/r3pc/:} 
[] !P3"A"{~>{} ~*{\alphanum} }
}
\ar@/^/|-{a_1} "A1";"A2"
\ar@/^/|-{a_2} "A2";"A3"
\ar@/^/|-{a_3} "A3";"A1"
\ar@/^/|-{b_3} "A1";"A3"
\ar@/^/|-{b_2} "A3";"A2"
\ar@/^/|-{b_1} "A2";"A1"
\endxy \\
\end{array}
\quad R=
\left\{\begin{array}{c}
a_1b_1 - b_3a_3, \\ 
a_2b_2 - b_1a_1,  \\
a_3b_3 - b_2a_2
\end{array}\right\}
\quad \mathbf d = \begin{matrix} 1\\ 1\quad 1 \end{matrix}.
\]

The chamber structure of the parameter space $\Theta \in \mathbb Q^2$ is as follows.
\[
\scalebox{1}{$
\def\objectstyle{\footnotesize}
\def\alphanum{\ifcase\xypolynode\or \or \theta_1=0 \or \theta_2=0 \or \or \or \fi}
\xy \xygraph{!{/r5pc/:} [] !P6"A"{~>{} ~*{\alphanum} 
}}
\ar@{-} "A1";"A4"
\ar@{-} "A2";"A5"
\ar@{-} "A3";"A6"
\ar@{}|-{C(s_1)} "A1";"A2"
\ar@{}^{C(1)} "A2";"A3"
\ar@{}|-{C(s_2)} "A3";"A4"
\ar@{}|-{C(s_2s_1)} "A4";"A5"
\ar@{}^{C(s_1s_2s_1)} "A5";"A6"
\ar@{}|-{C(s_1s_2)} "A6";"A1"
\endxy$}
{\theta_1+\theta_2=-\theta_0=0}
\]

First we consider $\mathcal M_1$. 
Then the exceptional set $E_1 \cup E_2$ is a chain of two $\mathbb P^1$'s and by Theorem \ref{thm:CB}, these are given as follows.
\[
\begin{array}{c}
E_1 = \{ M \in \s{\theta}{\Lambda}{,\mathbf d} \mid \Hom_{\Lambda}(S_1,M) \neq 0 \}
=\{ 
\def\objectstyle{\footnotesize}
\def\alphanum{\ifcase\xypolynode \or K \or K \or K \fi}
\xy 
\xygraph{
!{/r1.3pc/:} 
[] !P3"A"{~>{} ~*{\alphanum} }
}
\ar_{a} "A1";"A2"
\ar^{1} "A1";"A3"
\ar^{b} "A3";"A2"
\endxy
\mid (a,b) \in \mathbb P^1_K \}/\simeq,\\
E_2 = \{ M \in \s{\theta}{\Lambda}{,\mathbf d} \mid \Hom_{\Lambda}(S_2,M) \neq 0 \}
=\{ 
\def\objectstyle{\footnotesize}
\def\alphanum{\ifcase\xypolynode \or K \or K \or K \fi}
\xy 
\xygraph{
!{/r1.3pc/:} 
[] !P3"A"{~>{} ~*{\alphanum} }
}
\ar_{1} "A1";"A2"
\ar^{c} "A1";"A3"
\ar^{d} "A2";"A3"
\endxy
\mid (c,d) \in \mathbb P^1_K \},
\end{array}
\]
and the intersection of $E_1$ and $E_2$ is 
\[
\begin{array}{c}
E_1 \cap E_2 =\{ 
\def\objectstyle{\footnotesize}
\def\alphanum{\ifcase\xypolynode \or K \or K \or K \fi}
\xy 
\xygraph{
!{/r1.3pc/:} 
[] !P3"A"{~>{} ~*{\alphanum} }
}
\ar_{1} "A1";"A2"
\ar^{1} "A1";"A3"
\endxy\}
\end{array}.
\]
Note that we omit to write zero maps in each representations and actually consider isomorphism classes of them.
Pictorially $\mathcal M_1$ is described as follows where $(x,y) \neq (0,0)$ is a point in $\mathbb A^2$.

\begin{center}
\includegraphics[width=9cm]{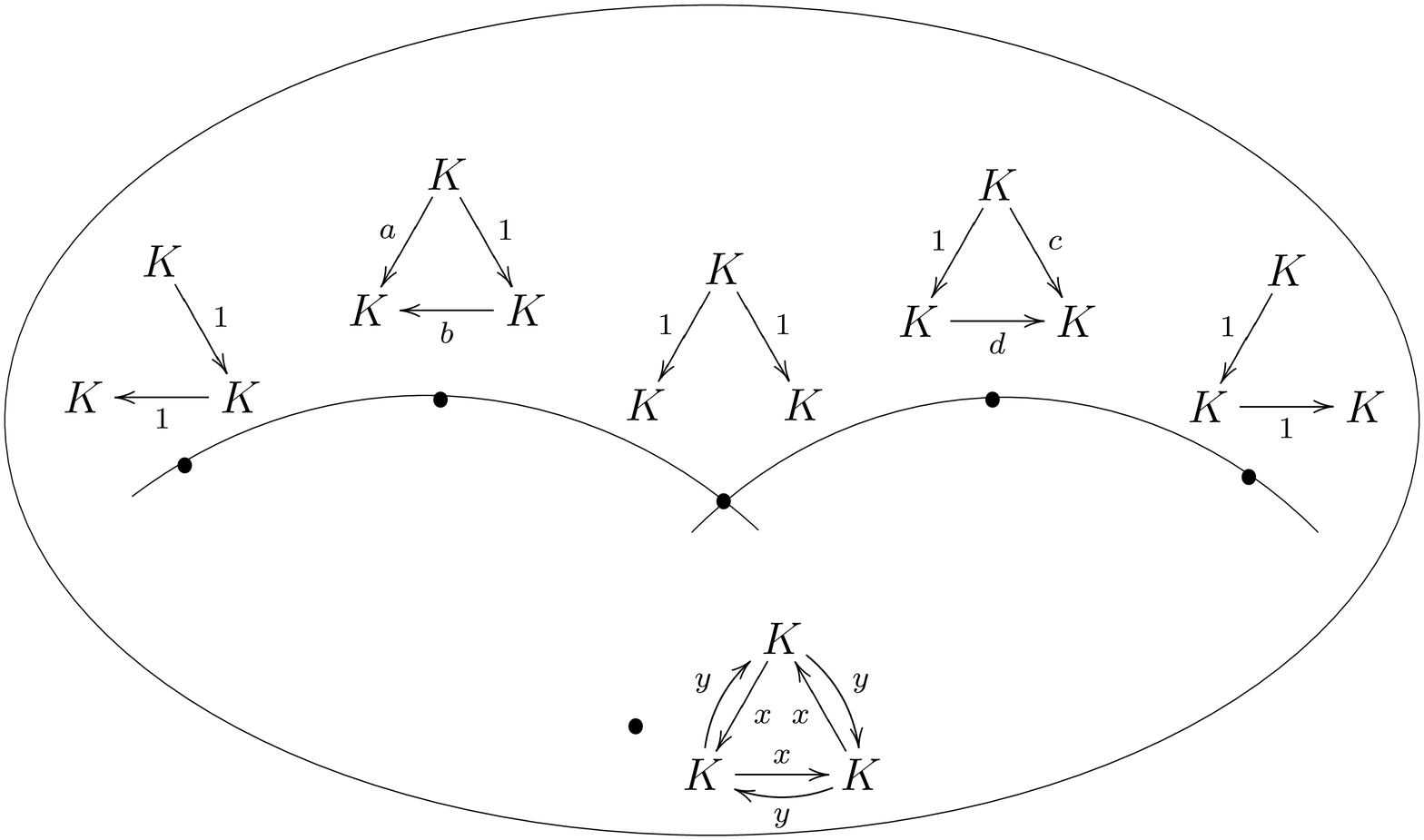}
\end{center}

Next we observe the relation between $\mathcal M_1$ and $\mathcal M_{s_1}$. 

\[
\begin{array}{c}
E_1^{s_1} 
=\{ 
\def\objectstyle{\footnotesize}
\def\alphanum{\ifcase\xypolynode \or K \or K \or K \fi}
\xy 
\xygraph{
!{/r1.3pc/:} 
[] !P3"A"{~>{} ~*{\alphanum} }
}
\ar^{a'} "A2";"A1"
\ar^{1} "A1";"A3"
\ar_{b'} "A2";"A3"
\endxy
\mid (a',b') \in \mathbb P^1_K \},
\end{array}
\]

\[
E_2^{s_1} 
= E_2 \setminus 
\{ \def\objectstyle{\footnotesize}
\def\alphanum{\ifcase\xypolynode \or K \or K \or K \fi}
\xy 
\xygraph{
!{/r1.3pc/:} 
[] !P3"A"{~>{} ~*{\alphanum} }
}
\ar_{1} "A1";"A2"
\ar^{1} "A1";"A3"
\endxy
\} 
\cup 
\{ \def\objectstyle{\footnotesize}
\def\alphanum{\ifcase\xypolynode \or K \or K \or K \fi}
\xy 
\xygraph{
!{/r1.3pc/:} 
[] !P3"A"{~>{} ~*{\alphanum} }
}
\ar_{1} "A1";"A2"
\ar^{1} "A2";"A3"
\endxy \} \\ 
=\{ 
\def\objectstyle{\footnotesize}
\def\alphanum{\ifcase\xypolynode \or K \or K \or K \fi}
\xy 
\xygraph{
!{/r1.3pc/:} 
[] !P3"A"{~>{} ~*{\alphanum} }
}
\ar_{d'} "A1";"A2"
\ar^{c'} "A1";"A3"
\ar^{1} "A2";"A3"
\endxy
\mid (c',d') \in \mathbb P^1_K \},
\]

and the intersection of $E_1^{s_1}$ and $E_2^{s_1}$ is 
\[
\begin{array}{c}
E_1^{s_1} \cap E_2^{s_1} =\{ 
\def\objectstyle{\footnotesize}
\def\alphanum{\ifcase\xypolynode \or K \or K \or K \fi}
\xy 
\xygraph{
!{/r1.3pc/:} 
[] !P3"A"{~>{} ~*{\alphanum} }
}
\ar^{1} "A1";"A3"
\ar^{1} "A2";"A3"
\endxy\}
\end{array}.
\]

Now $s_1\Delta = \{ -\e_1, \e_1+\e_2 \}$, and the dimension of $\Ext_{\Lambda}^1(I_1,S_1) \simeq S_1$ is $\e_1$.  

We express the exceptional curves on $\mathcal M_{s_1}$ by 
\[
E_1^{s_1} : 
\def\objectstyle{\footnotesize}
\def\alphanum{\ifcase\xypolynode \or \bullet \or \textcolor{blue}{\otimes} \or \bullet \fi}
\xy 
\xygraph{
!{/r2pc/:} 
[] !P3"A"{~>{} ~*{\alphanum} }
}
\ar "A2";"A1"
\ar "A2";"A3"
\ar "A1";"A3"
\endxy
\quad E_2^{s_1} : 
\def\alphanum{\ifcase\xypolynode \or \bullet \or \textcolor{red}{\oplus} \or \textcolor{red}{\oplus} \fi}
\xy 
\xygraph{
!{/r2pc/:} 
[] !P3"A"{~>{} ~*{\alphanum} }
}
\ar "A1";"A2"
\ar "A2";"A3"
\ar "A1";"A3"
\endxy
\]
where $\textcolor{blue}{\otimes}$ implies the quotient $\Ext_{\Lambda}^1(I_1,S_1)$ and $\textcolor{red}{\oplus} \longrightarrow \textcolor{red}{\oplus}$ the submodule $\Hom_{\Lambda}(I_1,S_2)$.

For all chambers, if we draw the exceptional curves by using the above expression, then the result is described in Figure \ref{figure2}.

\begin{figure}[htbp]
\begin{center}
\[
\begin{array}{|c|ccc|c|}\hline
\mathcal M_w & E_1^w & E_1^w \cap E_2^w & E_2^w & w\Delta \\\hline
\mathcal M_1 & 
\def\objectstyle{\footnotesize}
\def\alphanum{\ifcase\xypolynode \or \bullet \or \textcolor{red}{\oplus} \or \bullet \fi}
\xy 
\xygraph{
!{/r2pc/:} 
[] !P3"A"{~>{} ~*{\alphanum} }
}
\ar "A1";"A2"
\ar "A3";"A2"
\ar "A1";"A3"
\endxy &
\def\alphanum{\ifcase\xypolynode \or \bullet \or \textcolor{red}{\oplus} \or \textcolor{red}{\oplus} \fi}
\xy 
\xygraph{
!{/r2pc/:} 
[] !P3"A"{~>{} ~*{\alphanum} }
}
\ar "A1";"A2"
\ar "A1";"A3"
\endxy &
\def\alphanum{\ifcase\xypolynode \or \bullet \or \bullet \or \textcolor{red}{\oplus} \fi}
\xy \xygraph{
!{/r2pc/:} 
[] !P3"A"{~>{} ~*{\alphanum} }
}
\ar "A1";"A2"
\ar "A2";"A3"
\ar "A1";"A3"
\endxy
& \e_1, \e_2 \\
\mathcal M_{s_1} & 
\def\objectstyle{\footnotesize}
\def\alphanum{\ifcase\xypolynode \or \bullet \or \textcolor{blue}{\otimes} \or \bullet \fi}
\xy 
\xygraph{
!{/r2pc/:} 
[] !P3"A"{~>{} ~*{\alphanum} }
}
\ar "A2";"A1"
\ar "A2";"A3"
\ar "A1";"A3"
\endxy &
\def\alphanum{\ifcase\xypolynode \or \bullet \or \textcolor{red}{\oplus} \hspace{-0.8em}\textcolor{blue}{\otimes} \or \textcolor{red}{\oplus} \fi}
\xy 
\xygraph{
!{/r2pc/:} 
[] !P3"A"{~>{} ~*{\alphanum} }
}
\ar "A2";"A3"
\ar "A1";"A3"
\endxy &
\def\alphanum{\ifcase\xypolynode \or \bullet \or \textcolor{red}{\oplus} \or \textcolor{red}{\oplus} \fi}
\xy \xygraph{
!{/r2pc/:} 
[] !P3"A"{~>{} ~*{\alphanum} }
}
\ar "A1";"A2"
\ar "A2";"A3"
\ar "A1";"A3"
\endxy
& -\e_1, \e_1+\e_2\\
\mathcal M_{s_2} & 
\def\objectstyle{\footnotesize}
\def\alphanum{\ifcase\xypolynode \or \bullet \or \textcolor{red}{\oplus} \or \textcolor{red}{\oplus} \fi}
\xy 
\xygraph{
!{/r2pc/:} 
[] !P3"A"{~>{} ~*{\alphanum} }
}
\ar "A1";"A2"
\ar "A3";"A2"
\ar "A1";"A3"
\endxy &
\def\alphanum{\ifcase\xypolynode \or \bullet \or \textcolor{red}{\oplus} \or \textcolor{red}{\oplus} \hspace{-0.8em}\textcolor{blue}{\otimes} \fi}
\xy 
\xygraph{
!{/r2pc/:} 
[] !P3"A"{~>{} ~*{\alphanum} }
}
\ar "A1";"A2"
\ar "A3";"A2"
\endxy &
\def\alphanum{\ifcase\xypolynode \or \bullet \or \bullet \or \textcolor{blue}{\otimes} \fi}
\xy \xygraph{
!{/r2pc/:} 
[] !P3"A"{~>{} ~*{\alphanum} }
}
\ar "A1";"A2"
\ar "A3";"A2"
\ar "A3";"A1"
\endxy
& \e_1+\e_2, -\e_2 \\
\mathcal M_{s_2s_1} & 
\def\objectstyle{\footnotesize}
\def\alphanum{\ifcase\xypolynode \or \bullet \or \textcolor{blue}{\otimes} \or \textcolor{blue}{\otimes} \fi}
\xy 
\xygraph{
!{/r2pc/:} 
[] !P3"A"{~>{} ~*{\alphanum} }
}
\ar "A2";"A1"
\ar "A3";"A2"
\ar "A3";"A1"
\endxy &
\def\alphanum{\ifcase\xypolynode \or \bullet \or \textcolor{red}{\oplus} \hspace{-0.8em}\textcolor{blue}{\otimes} \or \textcolor{blue}{\otimes} \fi}
\xy 
\xygraph{
!{/r2pc/:} 
[] !P3"A"{~>{} ~*{\alphanum} }
}
\ar "A3";"A2"
\ar "A3";"A1"
\endxy &
\def\alphanum{\ifcase\xypolynode \or \bullet \or \textcolor{red}{\oplus} \or \bullet \fi}
\xy \xygraph{
!{/r2pc/:} 
[] !P3"A"{~>{} ~*{\alphanum} }
}
\ar "A1";"A2"
\ar "A3";"A2"
\ar "A3";"A1"
\endxy
& -\e_1-\e_2, \e_1 \\
\mathcal M_{s_1s_2} & 
\def\objectstyle{\footnotesize}
\def\alphanum{\ifcase\xypolynode \or \bullet \or \bullet  \or\textcolor{red}{\oplus} \fi}
\xy 
\xygraph{
!{/r2pc/:} 
[] !P3"A"{~>{} ~*{\alphanum} }
}
\ar "A2";"A1"
\ar "A2";"A3"
\ar "A1";"A3"
\endxy &
\def\alphanum{\ifcase\xypolynode \or \bullet \or \textcolor{blue}{\otimes} \or \textcolor{red}{\oplus} \hspace{-0.8em}\textcolor{blue}{\otimes} \fi}
\xy 
\xygraph{
!{/r2pc/:} 
[] !P3"A"{~>{} ~*{\alphanum} }
}
\ar "A2";"A1"
\ar "A2";"A3"
\endxy &
\def\alphanum{\ifcase\xypolynode \or \bullet \or \textcolor{blue}{\otimes} \or \textcolor{blue}{\otimes} \fi}
\xy \xygraph{
!{/r2pc/:} 
[] !P3"A"{~>{} ~*{\alphanum} }
}
\ar "A2";"A1"
\ar "A2";"A3"
\ar "A3";"A1"
\endxy
& \e_2, -\e_1-\e_2 \\
\mathcal M_{s_1s_2s_1} & 
\def\objectstyle{\footnotesize}
\def\alphanum{\ifcase\xypolynode \or \bullet \or \bullet \or \textcolor{blue}{\otimes} \fi}
\xy \xygraph{
!{/r2pc/:} 
[] !P3"A"{~>{} ~*{\alphanum} }
}
\ar "A2";"A1"
\ar "A3";"A2"
\ar "A3";"A1"
\endxy &
\def\alphanum{\ifcase\xypolynode \or \bullet \or \textcolor{blue}{\otimes} \or \textcolor{blue}{\otimes} \fi}
\xy 
\xygraph{
!{/r2pc/:} 
[] !P3"A"{~>{} ~*{\alphanum} }
}
\ar "A2";"A1"
\ar "A3";"A1"
\endxy &
\def\alphanum{\ifcase\xypolynode \or \bullet \or \bullet \or \textcolor{blue}{\otimes}  \or \bullet \fi}
\xy \xygraph{
!{/r2pc/:} 
[] !P3"A"{~>{} ~*{\alphanum} }
}
\ar "A2";"A1"
\ar "A2";"A3"
\ar "A3";"A1"
\endxy
& -\e_2, -\e_1 \\\hline
\end{array}
\]
\caption{Characterization of exceptional curves}
\label{figure2}
\end{center}
\end{figure}


\end{document}